\documentclass[12pt]{amsart}

\usepackage[a4paper,margin=1in]{geometry}
\usepackage{amsfonts,amssymb,amsmath,amsthm}
\usepackage{graphicx}
\usepackage{psfrag}
\usepackage{calc}
\usepackage{epic}
\usepackage{hyperref}
\usepackage{float}
\usepackage{latexsym}
\usepackage{color}
\usepackage{subfigure,dsfont}
\usepackage{todonotes}
\usepackage[cal=pxtx,scr=boondoxo]{mathalfa}

\newtheorem{theorem}{Theorem}
\newtheorem{lemma}[theorem]{Lemma}
\newtheorem{proposition}[theorem]{Proposition}
\newtheorem{corollary}[theorem]{Corollary}
\newtheorem{remark}[theorem]{Remark}
\newtheorem{definition}[theorem]{Definition}

\newtheorem{condition}[theorem]{Condition}

\newcommand{\be}{ \begin{equation}}
\newcommand{\ee}{\end{equation}}
\newcommand{\ben}{ \begin{equation*}}
\newcommand{\een}{\end{equation*}}

\def\E{{\mathbb E}}
\def\P{{\mathbb P}}

\def\root{{\emptyset}}

\newcommand{\e}{{\rm e}}
\newcommand{\CMnd}{{\rm CM}_n(\boldsymbol{d})}
\newcommand{\CMnD}{{\rm CM}_n(\boldsymbol{D})}
\newcommand{\bfdit}{\boldsymbol{d}}
\newcommand{\Ver}{U}

\newcommand{\eqn}[1]{\begin{equation}#1\end{equation}}
\newcommand{\eqan}[1]{\begin{align}#1\end{align}}

\newcommand{\op}{o_{\sss {\mathbb P}}}
\newcommand{\convp}{\stackrel{\sss {\mathbb P}}{\longrightarrow}}
\newcommand{\convd}{\stackrel{\sss d}{\longrightarrow}}
\newcommand{\convas}{\stackrel{\sss a.s.}{\longrightarrow}}

\newcommand{\prob}{\mathbb P}
\newcommand{\expec}{\mathbb E}

\newcommand{\SWG}{{\sf SWG}}
\newcommand{\BP}{{\sf BP}}
\newcommand{\Wcal}{\mathcal{W}}
\newcommand{\Mcal}{\mathcal{M}}
\newcommand{\barNt}{\bar{N}^{\sss {\mathcal{F}}}_n(t)}
\newcommand{\msg}{\mathscr{g}}

\def\1{{\mathchoice {1\mskip-4mu\mathrm l}      
{1\mskip-4mu\mathrm l}
{1\mskip-4.5mu\mathrm l} {1\mskip-5mu\mathrm l}}}
\newcommand{\indic}[1]{\1_{\{#1\}}}
\newcommand{\indicwo}[1]{\1_{#1}}

\newcommand{\sss}{\scriptscriptstyle}
\newcommand{\nn}{\nonumber}
\newcommand{\vep}{\varepsilon}

\renewcommand{\root}{\varnothing}
\newcommand{\Time}{T}

\DeclareSymbolFont{extraup}{U}{zavm}{m}{n}
\DeclareMathSymbol{\varheart}{\mathalpha}{extraup}{86}
\DeclareMathSymbol{\vardiamond}{\mathalpha}{extraup}{87}
\newcommand{\ensymboldefinition}{\blacktriangleleft}

\newcommand{\seva}[1]{\todo[inline, color=lime]{Seva: #1}}

\title{It is hard to kill fake news}

\author{Remco van der Hofstad}\address{Eindhoven University of Technology, P.O. Box 513, 5600 MB Eindhoven, The Netherlands. E-mail: rhofstad@win.tue.nl}
\author{Seva Shneer}\address{School of MACS, Heriot-Watt University, Edinburgh, EH14 4AS, United Kingdom. E-mail: V.Shneer@hw.ac.uk}

\begin{document}
\maketitle

\begin{abstract} 
We study a model for the spread of fake news, where first a piece of fake news is spread from a location in a network, followed by a correction to the news. We assume that both the fake as well as correct news travel as first-passage percolations or SI epidemics with i.i.d.\ traversal times, possibly with different distributions and dependence. 

We make the (hopeful) assumption that once a vertex in the network hears the correct news, they are immediately convinced that this is indeed correct, and continue to spread the correct news. Even in this optimistic scenario, our main results show that it is very difficult to remove the fake news from the network, even when the correct news would spread much faster than the fake news. We show this on the configuration model, a model that has gained enormous popularity as a simple, yet flexible, model for real-world networks. 

The crux of the proof is the realization that this problem on a branching process tree (which is known to be the local limit of the configuration model) can be  reformulated in terms of the maximum of a branching random walk, a topic that has attracted considerable attention in the past decade. We then extend the results to the graph setting using couplings to branching processes, local convergence and detailed estimates on first-passage percolation on random graphs as derived in \cite{BhaHofHoo17}. Remarkably, despite the fact that our proofs for the configuration model rely on its local branching process structure, the condition for strong survival on  a finite number of generations of a branching process tree is {\em different} from that on the configuration model.
\end{abstract}

\section{Introduction}
\label{sec-intro}

%
%

Fake news has a disrupting effect on society. It has influenced the American elections in 2016, see \cite{AllGen17} for an extensive survey of the literature, as well as the Brexit referendum in the United Kingdom, see \cite{Hol21}.  We refer also to \cite{VosRoyAra18} for an extensive study of the spread of fake news on social media with results echoing the title of our paper. At present, a lot of research is conducted, particularly about how one can detect fake news automatically through clever algorithms \cite{ConRubChe15}.

Fake news exists of various types \cite{Lazer_etal18}, where ``fake news'' is defined to be ``fabricated information that mimics news media content in form but not in organisational process or intent'', while `` fake-news outlets, in turn, lack the news media's editorial norms and processes for ensuring the accuracy and credibility of information.'' The following distinctions are being made there: ``Fake news overlaps with other information disorders, such as misinformation (false or misleading information) and disinformation (false information that is purposely spread to deceive people).''

Sometimes, however, incorrect news is spread with no harmful intention, by parties with the best possible intentions. It is such news that we aim to model in this paper. We start from a rather optimistic perspective, where a well-intentioned party (such as a reliable news agency) has spread a piece of misinformation, but then realizes that it is false or misleading and retracts it. This retraction then also spreads through the network, possibly through different paths. We are optimistic, as we assume that a person who has heard the fake news, upon hearing the correction, is immediately convinced by it. In reality, we know that this is not the case, again see \cite{AllGen17}. However, even in this optimistic scenario, the society's network structure can make the fake news rather persistent. In our model, though, eventually the correct news will push out the fake news completely. 

The main innovations of this paper are that we define several ways how fake news can survive for a long time, and give conditions for them to occur on the configuration model and branching process trees. More precisely, fake news can {\em weakly survive}, meaning that the number of people that are exposed to it grows arbitrarily large with positive probability. A stronger notion of fake news survival is {\em strong survival}, in which a positive proportion of the population is with positive probability exposed to the fake news. The models for the spread of the news rely on {\em first-passage percolation}, a standard model for the spread of rumours that has attracted considerable attention in the random graph community. In one of the most popular network models, the configuration model, we then give rather precise conditions for weak and strong survival to occur. Even though these conditions do not cover all possible scenarios, they do cover most of them. We also investigate what happens when spreading fake news from the root of a branching process tree, which is relevant since the local limit of the configuration model is a (unimodular) branching process. Remarkably, the conditions for strong survival are {\em different} in these two settings, while the conditions for weak survival do agree. This shows that strong survival of fake news is highly non-local.

The remainder of this paper is organized as follows. In Section \ref{sec-mod-res}, we present the precise model, our results and a discussion of them. In Section \ref{sec-tree}, we present our results when applied to the source of the news being the root of a finite number of generations of a branching process. Finally, in Section \ref{sec-CM}, we present our proofs for the configuration model.

\section{Models, notation and results}
\label{sec-mod-res}

\subsection{The model}
\label{sec-mod} 
In this section, we define our model. We start by defining the precise setting of the fake news spreading, followed by the definition of the configuration model, which is the graph on which our information diffuses.
\medskip

\paragraph{\bf The fake news model.}
For now, assume that we are on a general graph $G=(V(G),E(G))$. Since our graphs are finite, we assume that $|V(G)|=n$ and write $V(G)=[n]=\{1, \ldots, n\}$. We are interested in large graph asymptotics, so we often take $n\rightarrow \infty$. Later we will work on a configuration model random graph. Denote by $L_e^{\sss \mathcal{F}}$ (resp.\ $L_e^{\sss \mathcal{R}}$) the traversal time of edge $e$ by fake (resp. correct) news. We will assume that each of these collections consists of i.i.d.\ random variables and will denote by $L^{\sss \mathcal{F}}$ and $L^{\sss \mathcal{R}}$ generic random variables having the corresponding distributions $F^{\sss \mathcal{F}}$ and $F^{\sss \mathcal{R}}$. An important special case consists of exponential traversal times, for which we denote the corresponding parameters by $\mu^{\sss \mathcal{F}}$ and $\mu^{\sss \mathcal{R}}$. We emphasize that $L^{\sss \mathcal{F}}$ and $L^{\sss \mathcal{R}}$ may be dependent, while $((L^{\sss \mathcal{F}}, L^{\sss \mathcal{R}}))_{e \in E(G)}$ is a collection of i.i.d.\ vectors.


\begin{definition}[Start of news epidemics]
\label{def-start-delay}
{\rm We assume that the fake and correct news start at a source that is chosen uniformly at random from the vertex set $[n]$. For the spread of the fake and correct news epidemics, we will consider two possible starting cases. In both cases, fake news is not killed at the source. In the first, we assume that both the fake- and the correct-news epidemics start at the same time. In the second case, we introduce a {\em delay} $d$ that describes how much later the correct news starts compared to the fake news, effectively making it easier for the fake news to survive. Since adding a delay makes the notation more cumbersome, we focus on the setting without delay in our proof, and only highlight the differences. We frequently use that the setting without delay can be obtained by taking $d=0$.}\hfill$\ensymboldefinition$
\end{definition} 

To avoid trivialities, we assume that
	\begin{equation} 
	\label{eq:feasibility}
	\P(L^{\sss \mathcal{F}} < L^{\sss \mathcal{R}}) > 0.
	\end{equation}

This model is a particular case of the so-called predator-prey models, usually studied on lattices \cite{Bordenave08,Hag98,Kord05,KordLa05,Kort16}. In \cite{Kord05} the exact model we consider was referred to as the escape model and was studied on a regular tree under assumptions equivalent to both $L^{\sss \mathcal{F}}$ and $L^{\sss \mathcal{R}}$ having Exponential distributions. We will show (see Remark \ref{rem-exp}) that our results recover those of \cite{Kord05}. The author of \cite{Kort16} draws a connection between predator-prey models and branching random walks - we exploit the same connection in our paper.
	
This model is also closely related to {\em first-passage percolation}, a model that has attracted considerable attention in the recent years, see \cite{BarHofKom17, Bham08, BhaHofHoo09a, BhaHofHoo09b, BhaHofHoo12, BhaHofHoo17} and the references therein, and \cite[Chapter 2]{Hofs20} for an extensive survey. The model is also closely related to competition on (random) graphs, see e.g., \cite{AhlDeiJan17, BarHofKom15, DeiHof16, DeiHofSfr22, HofKom15}.
\smallskip

Denote by $N^{\sss \mathcal{F}}_n$ the number of vertices reached by the fake-news epidemic before being reached by the correct news epidemic. The main aim of this paper is to establish when the fake news survives. For this, we define three notions of survival:

\begin{definition}[Weak survival]
\label{def-weak-surv}
{\rm We say that the fake news survives {\em in the weak sense} when}
	\eqn{
	\liminf_{K\rightarrow \infty}\liminf_{n\rightarrow \infty}\prob(N^{\sss \mathcal{F}}_n\geq K)>0.
	}
\hfill$\ensymboldefinition$
\end{definition}

When the fake news survives in the weak sense, then, with a uniformly positive probability, it reaches an arbitrarily large number of vertices. This should be contrasted to {\em strong survival}, for which, with uniformly positive probability, the fake news reaches a positive proportion of the vertices:
\begin{definition}[Strong survival]
\label{def-strong-surv}
{\rm We say that the fake news survives {\em in the strong sense} when}
	\eqn{
	\liminf_{\eta\searrow 0}\liminf_{n\rightarrow \infty}\prob(N^{\sss \mathcal{F}}_n\geq \eta n)>0.
	}
\hfill$\ensymboldefinition$
\end{definition}

Of course, weak survival is quite weak, while strong survival possibly too strong. The definition below quantifies intermediate survival phases:

\begin{definition}[Intermediate survival]
\label{def-intermediate-surv}
{\rm We say that the fake news {\em lower survives} {\em in the $k_n$-intermediate sense} for some $k_n\rightarrow \infty$ when there exists a constant $C>1$ such that
	\eqn{
	\liminf_{n\rightarrow \infty}\prob(N^{\sss \mathcal{F}}_n\geq k_n/C)>0.
	}
We say that the fake news {\em upper dies} {\em in the $k_n$-intermediate sense} for some $k_n\rightarrow \infty$ when $(N^{\sss \mathcal{F}}_n/k_n)_{n\geq 1}$ is a tight sequence of random variables.
We say that the fake news survives {\em in the $k_n$-intermediate sense} if it lower survives and upper dies in the $k_n$-intermediate sense.}\hfill$\ensymboldefinition$ 
\end{definition}

In terms of Definition \ref{def-intermediate-surv}, strong survival is equivalent to $n$-intermediate survival. Weak survival is not quite comparable to intermediate survival, due to the related upper bound on $N^{\sss \mathcal{F}}_n$. The above definitions make sense on any finite graph. We next specify the type of graphs on which we consider our information diffusion process.
\medskip

\paragraph{\bf Configuration model.}
Fix an integer $n$ that is the number of vertices in the random graph and denote the set $\{1,2,\ldots n\}$ by $[n]$. Consider a sequence of degrees $\bfdit=(d_i)_{i\in[n]}$. 
Without loss of generality, we assume that $d_j\geq 1$ for all $j\in [n]$, since when $d_j=0$, vertex $j$ is isolated and can be removed from the graph. 
Furthermore, we assume that the total degree
    	\eqn{
    	\ell_n=\sum_{j\in [n]} d_j
    	}
is even. When $\ell_n$ is odd we increase the degree $d_n$ by 1.  For $n$ large, this will not change the results and we will therefore
ignore this effect. We wish to construct a (multi)graph such that $\bfdit=(d_i)_{i\in[n]}$ are the degrees of the $n$ vertices. 
To construct the multigraph where vertex $j$ has degree $d_j$ for all $j\in [n]$, we have $n$ separate vertices and incident to vertex $j$, we have $d_j$ half-edges. Every half-edge needs to be connected to another half-edge to form an edge, and by forming all edges we build the graph. For this, the half-edges are numbered in an arbitrary order from $1$ to $\ell_n$. We start by randomly connecting the first half-edge with one of the $\ell_n-1$ remaining half-edges. Once paired, two half-edges form a single edge of the multigraph, and the half-edges are removed from the list of half-edges that need to be paired. Hence, a half-edge can be seen as the left or the right half of an edge. We continue the procedure of randomly choosing and pairing the half-edges until all half-edges are connected, and, following Bollob\'as \cite{Boll80b}, we call the resulting graph the {\it configuration model with degree sequence $\bfdit$}, abbreviated as $\CMnd$. 
See \cite[Chapter 7]{Hofs17} for an extensive introduction to this model.

We impose certain \emph{regularity conditions} on the degree sequence $\bfdit$. In order to state these assumptions, we introduce some notation. We denote the degree of a uniformly chosen vertex $\Ver$ in $[n]$ by $D_n=d_{\Ver}$. The random variable $D_n$ has distribution function $F_n$ given by
    \eqn{
    \label{def-Fn-CM}
    F_n(x)=\frac{1}{n} \sum_{j\in [n]} \indic{d_j\leq x},
    }
which is the {\em empirical distribution of the degrees.}
We assume that the vertex degrees satisfy the following \emph{regularity conditions:}

\begin{condition}[Regularity conditions for vertex degrees]
\label{cond-degrees-regcond}
~\\
{\bf (a) Weak convergence of vertex degrees.}
There exists a distribution function $F$ such that, as $n\rightarrow \infty$,
    \eqn{
    \label{Dn-weak-conv}
    D_n\convd D,
    }
where $D_n$ and $D$ have distribution functions $F_n$ and $F$, respectively.\\
Equivalently, for any $x$, 
    \eqn{
    \label{conv-Fn-CM}
    \lim_{n\rightarrow \infty} F_n(x)=F(x).
    }
Further, we assume that $F(1)=0$, i.e., $\prob(D\geq 2)=1$.\\
{\bf (b) Convergence of average vertex degrees.} As $n\rightarrow \infty$,
    \eqn{
    \label{conv-mom-Dn}
    \expec[D_n]\rightarrow \expec[D] \in (0,\infty),
    }
where $D_n$ and $D$ have distribution functions $F_n$ and $F$ from part (a), respectively.\\
{\bf (c) Convergence of second moment vertex degrees.} As $n\rightarrow \infty$,
    \eqn{
    \label{conv-sec-mom-Dn}
    \expec[D_n^2]\rightarrow \expec[D^2] < \infty,
    }
where again $D_n$ and $D$ have distribution functions $F_n$ and $F$ from part (a), respectively.
\end{condition}

Two canonical examples of such degree sequences are when we take $d_i=[1-F]^{-1}(i/n)$, where $F$ is the distribution function of an integer-valued random variable, and when $(d_i)_{i\in[n]}$ constitutes a {\em realization} of an i.i.d.\ sequence of random variables with distribution function $F$. We denote the asymptotic degrees distribution by $\boldsymbol{p} = (p_k)_{k\geq 1}$ with $p_k=\prob(D=k)$. See \cite[Chapter 7]{Hofs17} for an extensive discussion of the configuration model and the Degree Regularity Condition \ref{cond-degrees-regcond}.

We let
$$
\nu = \E[D^*]-1,
$$
where $D^*$ is the sized-biased version of $D$.

\subsection{Main results}
\label{sec-res}
Let $\varphi^{\sss \mathcal{R}}(s) = \E[\e^{sL^{\sss \mathcal{R}}}]$, $\varphi^{\sss \mathcal{F}}(s) = \E[\e^{sL^{\sss \mathcal{F}}}]$ and
	\begin{equation}
	\psi(s) = \E[\e^{s(L^{\sss \mathcal{R}} - L^{\sss \mathcal{F}})}].
	\end{equation}

\begin{theorem}[Weak survival]
\label{thm-weak-CM}
Condition on the news originating in the giant component.
\begin{itemize} 
\item[(i)] If $\E[L^{\sss \mathcal{R}}] > \E[L^{\sss \mathcal{F}}]$, then fake news survives weakly.
\item[(ii)] Assume $\E[L^{\sss \mathcal{R}}] \le \E[L^{\sss \mathcal{F}}]$ and assume, in addition, that $\psi(s) < \infty$ for some $s>0$. Assume further that there exists $0 < h < s$ such that $\psi'(h) = 0$. Denote  $\rho = \psi(h)$. If $\rho > 1/\nu$, then fake news weakly survives. If $\rho \leq 1/\nu$, then fake news does not survive in the weak sense.
\item[(iii)] Assume $\E[L^{\sss \mathcal{R}}] \le \E[L^{\sss \mathcal{F}}]$. If $\psi(s) = \infty$ for all $s>0$, then fake news weakly survives.
\end{itemize}
\end{theorem}

\begin{remark}[Scenario not covered by Theorem \ref{thm-weak-CM}]
\label{rem-pathology}
{\rm We note that the only scenario not covered by Theorem \ref{thm-weak-CM} is the pathological case when $\E[L^{\sss \mathcal{R}}] \le \E[L^{\sss \mathcal{F}}]$, there exists $s > 0$ such that $\psi(s) < \infty$, but the equation $\psi'(h) = 0$ does not have a solution. Clearly, $\psi(0)=1$, $\psi'(0)=\E[L^{\sss \mathcal{F}}]-\E[L^{\sss \mathcal{R}}]\leq 0$ and $\psi(s)$ is convex, so $h$ exists for vast majority of distributions. This poses a well-known difficulty in studying passage times for random walks, and very few results are known when $h$ does not exist. As we rely on asymptotics of random-walks passage times in our analysis, we exclude this scenario.}\hfill$\ensymboldefinition$
\end{remark}

We prove Theorem \ref{thm-weak-CM} in Section \ref{sec-weak-CM}, by using that the configuration model converges locally weakly to a branching process. This implies that the fake news dynamics can be related to that on a branching process, a setting that we study in detail in Section \ref{sec-tree}.

\begin{theorem}[Strong survival]
\label{thm-strong-CM}
Consider $\CMnd$ where the degree distribution satisfies Conditions \ref{cond-degrees-regcond}(a)-(c). Assume further that
	\eqn{
	\label{cond-D2logD}
	\expec[D_n^2 \log(D_n)]\rightarrow \expec[D^2\log(D)]<\infty.
	}
Let $\lambda^{\sss \mathcal{R}}$ and $\lambda^{\sss \mathcal{F}}$ denote the Malthusian parameters of first-passage percolation with offspring distribution $D^\star-1$ and lifetimes $L^{\sss \mathcal{R}}$ and $L^{\sss \mathcal{F}}$, respectively. If $\lambda^{\sss \mathcal{R}}<\lambda^{\sss \mathcal{F}}$, then fake news survives in the strong sense. If, on the other hand, $\lambda^{\sss \mathcal{R}}>\lambda^{\sss \mathcal{F}}$, then fake news does not survive in the strong sense.
\end{theorem}

\begin{remark}[Malthusian parameters]
{\rm We recall that the Malthusian parameters $\lambda^{\sss \mathcal{R}}$ and $\lambda^{\sss \mathcal{F}}$ are the unique positive solutions to the equations
$$
\varphi^{\sss \mathcal{R}}(-s) = \expec[\e^{-sL^{\sss \mathcal{R}}}] = \frac{1}{\nu}
\qquad
\text{and}
\qquad
\varphi^{\sss \mathcal{F}}(-s) = \expec[\e^{-sL^{\sss \mathcal{F}}}] = \frac{1}{\nu},
$$
respectively.}\hfill$\ensymboldefinition$
\end{remark}

\begin{remark}[Relation Theorems \ref{thm-weak-CM} and \ref{thm-strong-CM}]
\label{rem-weak-vs-strong}
{\rm We note that as strong survival implies weak survival, it should hold that $\rho > 1/\nu$ when $\lambda^{\sss \mathcal{R}}<\lambda^{\sss \mathcal{F}}$ and the assumptions of case (ii) of Theorem \ref{thm-weak-CM} hold. We leave the proof of this to a later stage (see Remark \ref{rem-weak-vs-strong-1}) when we introduce the necessary notation and prove some useful results, and now only give the proof in a particular case when $L^{\sss \mathcal{F}}$ and $L^{\sss \mathcal{R}}$ are negatively associated (including the case when they are independent). First, it is clear that it is sufficient to show that $h < \lambda^{\sss \mathcal{F}}$ as in that case clearly
	\eqn{
	\label{strong-weak-2}
	\rho = \psi(h) 
	 \ge \varphi^{\sss \mathcal{F}}(-h) > \varphi^{\sss \mathcal{F}}(-\lambda^{\sss \mathcal{F}}) = 1/\nu.
	}
This condition holds when $L^{\sss \mathcal{F}}$ and $L^{\sss \mathcal{R}}$ are negatively associated, as then
	\eqn{
	\label{strong-weak-1}
	\psi(\lambda^{\sss \mathcal{F}}) 
	\ge \varphi^{\sss \mathcal{R}}(\lambda^{\sss \mathcal{F}}) \varphi^{\sss \mathcal{F}}(-\lambda^{\sss \mathcal{F}}) 
	> \varphi^{\sss \mathcal{R}}(\lambda^{\sss \mathcal{R}})  \varphi^{\sss \mathcal{F}}(-\lambda^{\sss \mathcal{F}}) 
	= \frac{1}{\nu} \varphi^{\sss \mathcal{R}}(\lambda^{\sss \mathcal{R}}) 
	\ge \frac{1}{\nu} \frac{1}{\varphi^{\sss \mathcal{R}}(-\lambda^{\sss \mathcal{R}})} 
	= 1,
	}
where the last inequality is due to the convexity of the function $1/x$. As $\psi(0)=1$, the above implies that $h < \lambda^{\sss \mathcal{F}}$.
}\hfill$\ensymboldefinition$
\end{remark}

In the proof of Theorem \ref{thm-strong-CM}, we heavily rely on the asymptotics of first-passage percolation on the CM, as studied in \cite{BhaHofHoo17}. It is there that the condition \eqref{cond-D2logD} is crucially used. In terms of the branching process local limit of the configuration model, \eqref{cond-D2logD} is equivalent to the $X\log{X}$ condition that is needed for the Kesten-Stigum theorem to apply (see e.g., \cite{JagNer84, LyoPemPer95}). The strong survival results are proved in Section \ref{sec-strong-CM}, while the results on non-strong survival are proved in Section \ref{sec-no-strong-CM}. Interestingly, the strong survival results do not relate easily to strong survival on a truncation of an infinite branching process tree:

\begin{remark}[Relation Theorem \ref{thm-strong-CM} and strong survival on a tree]
{\rm Due to the close relation between the configuration model and branching processes, one might suspect that strong survival occurs \emph{precisely} when strong survival occurs from the root of a finite number of generations of a discrete branching process. Such a statement is true for weak survival, see Theorem \ref{thm-weak-CM}, and compare this to the corresponding result for branching processes in Theorem \ref{thm:main_tree} below. For strong survival, however, this analogy is surprisingly {\em false}. Indeed, Theorem \ref{thm:strong_tree} below shows that, conditionally on survival, fake news strongly survives on a finite number of generations of a branching process if and only if $\E(L^{\sss \mathcal{R}}) > \E(L^{\sss \mathcal{F}})$. This can be proved rather easily by a relatively simple random walk argument. The fact that we need the more involved condition on the Malthusian parameters in Theorem \ref{thm-strong-CM} is due to the fact that a {\em typical} vertex that is reached by fake news, when $\lambda^{\sss \mathcal{R}}<\lambda^{\sss \mathcal{F}}$, is reached through a very long path, and not a path that stays close to the root, as is the main contribution for branching processes in Theorem \ref{thm:strong_tree}. Thus, strong survival is highly non-local, contrary to e.g., the size of the giant (see e.g., \cite{Hofs21b}).
}\hfill$\ensymboldefinition$
\end{remark}

We continue by studying the case where the continuous-time branching process approximation of the random graph is {\em explosive}, which may occur in the setting where the degrees have infinite variance. This infinite-variance-degree setting has been investigated vigorously in the literature \cite{BarHofKom15, BarHofKom17, BhaHofHoo09b, DeiHof16, JorKom20}:

\begin{theorem}[Strong survival: explosive setting]
\label{thm-strong-CM-explos}
Consider $\CMnD$ where the degrees are i.i.d.\ random variables with $\prob(D\geq 2)=1$ and, as $k\rightarrow \infty$,
	\eqn{
	\label{def-inf-var}
	\prob(D> k) =(1+o(1)) \frac{c}{k^{\tau-1}}.
	}
Assume that the continuous-time branching process with offspring distribution $D^\star-1$ and edge weight $L^{\sss\mathcal{F}}$ is explosive. Then the fake news survives in the strong sense.
\end{theorem}

In the proof of Theorem \ref{thm-strong-CM-explos} in Section \ref{sec-strong-CM-explos}, we heavily rely on the asymptotics of first-passage percolation on the CM in the explosive setting, as studied in \cite{BarHofKom15, BarHofKom17, BhaHofHoo09b, DeiHof16}. In particular, we prove an {\em epidemic curve} result, in Proposition \ref{prop-degree-time-FPP} below, in which we show that (a) with positive probability, the fake news reaches the hubs before the correct news, and (b) almost instantly after the fake news reaches the hubs, it reaches a positive proportion of the vertices, effectively blocking the correct news from progressing.

\subsection{Discussion and open problems}
\label{sec-disc}
In this section, we discuss our results and state open problems. 
\medskip



\paragraph{\bf Equal Malthusian parameters.} Theorem \ref{thm-strong-CM} leaves open the interesting case where the Malthusian parameters are equal, i.e., when $\lambda^{\sss \mathcal{R}}=\lambda^{\sss \mathcal{F}}$. There, we believe that lower order effects can play a crucial role. Indeed, the proof relies on finite $n$ approximations $\lambda^{\sss \mathcal{R}}_n$ and $\lambda^{\sss \mathcal{F}}_n$ of $\lambda^{\sss \mathcal{R}}$ and $\lambda^{\sss \mathcal{F}}$, and even when their limits $\lambda^{\sss \mathcal{R}}$ and $\lambda^{\sss \mathcal{F}}$ agree, their prelimits can still be far enough apart to make the system strongly survive. However, for example for exponential edge weights with equal parameter, it is not clear what will happen, since in this case even $\lambda^{\sss \mathcal{R}}_n=\lambda^{\sss \mathcal{F}}_n$. This is a highly interesting topic for future research.
\medskip

\paragraph{\bf Several and distinct starting points.} We have assumed the simplest possible setting, where the correct and fake news have the same source. One may wonder how the results change when they start from different locations, or even when there are several starting points for the fake news and/or the correct news. This problem is closely related to competition models (see, e.g., \cite{BarHofKom15, DeiHof16, HofKom15}).

\medskip

\paragraph{\bf Scale of the fake news set in weak, but not strong, survival.} Due to Remark \ref{rem-weak-vs-strong}, Theorems \ref{thm-weak-CM} and \ref{thm-strong-CM} leave a gap where the processes are weakly surviving, but not strongly surviving. In this case, it would be of interest to investigate the scaling of the number of vertices that hear the fake news. Some weak bounds may be extracted from the proof in Section  \ref{sec-no-strong-CM}, which indicates that the fake news may be at most $n^{\lambda^{\sss \mathcal{F}}/\lambda^{\sss \mathcal{R}}+o(1)}$ intermediate surviving. We expect that this is sharp. We currently miss the appropriate and interesting lower bound.
\medskip

\paragraph{\bf Interpolation competition and fake news.} Consider an extension of the model presented here, where correct and fake news spread to unoccupied vertices as in first-passage percolation. When a vertex is occupied by correct news, it will remain to be so forever, and continue to spread correct news. When a vertex is occupied by fake news, (say) the first time correct news attempts to occupy it, it will succeed with probability $p$, and fail with probability $1-p$. The model with $p=0$ is the classic competition model, where the vertices after being occupied never change type. Then model with $p=1$ is our model of fake news spread, since a vertex will simply {\em always} accept correct news, so that correct news spreads like first-passage percolation, while fake news gets replaced by correct news. The model with $p\in(0,1)$ thus combines competition with fake news spread, and might be more realistic for fake news or misinformation spreading than ours.

\medskip
\paragraph{\bf Other spreading mechanisms.} Extending the above discussion, we have discussed one possible way how the spreading of fake and correct news can take place. However, many other possible spreading processes could be invented that model the spread, or adaption, of fake news more realistically. We believe that the applied probability community can play an important role in this societally central problem, both in terms of proposing relevant models, as well as studying the behavior in these models.

\medskip
\paragraph{\bf Models with community structure.} The configuration model is {\em locally tree-like}, meaning that typical neighborhoods look like trees. In particular, such models do not display a community structure. Communities, and the fact that people tend to communicate primarily with people having similar opinions, are claimed to be one of the main reasons why fake news can stick around for such an extended period in so-called echo chambers. It would thus be of interest to study our fake news model on a graph with community structures, such as the hierarchical configuration model \cite{HofLeeSte17}. In particular, Stegehuis et al.\ \cite{SteHofLee16a} show that such a model is much better at describing information diffusion on real-world networks than the configuration model itself. 

\section{Tree model}
\label{sec-tree}

We start by considering the model of competition on a tree. This will help us understand the key characteristics of the model, as well as provide a basis for the analysis on a random graph. Note that many random graph models converge in the local sense to branching processes (see e.g., \cite[Chapters 2-5]{Hofs18} for a detailed discussion of local convergence, as well as the convergence proof for many random graph models).

Let us denote the neighbours of the root as $\root 1,\ldots,\root C_{\root}$, where $C_{\root}$ is the number of children of the root, and we refer to these vertices as the first generation. Continuing this, for a vertex $\root i_1\cdots i_{l-1}$, denote its children (in the $l$th generation) as $\root i_1 \cdots i_{l-1} 1,\ldots,\root i_1 \cdots i_{l-1} C_{\root i_1\cdots i_{l-1}}$. In general, denote the number of offspring of a vertex $w$ by $C_w$. We assume that these random variables are i.i.d.\ over the set of vertices and also independent of traversal times. We denote $\nu = \E(C_w)$ and assume $\nu > 1$ so that there is a positive probability for the branching process to survive. Note that we do not assume that $\nu < \infty$. All probabilities and expectations are to be understood as conditional on this survival. We assume that both epidemics start at the root, and correct news either starts after a positive fixed delay $d$, or without a delay and with the fake news not killed at the root (recall Definition \ref{def-start-delay}).

We adapt the definitions in Definitions \ref{def-weak-surv}, \ref{def-strong-surv} and \ref{def-intermediate-surv}, by thinking of them as applying to the graph of the first $k$ generations  and with the starting point being the root of the branching process tree. Denote the numbers of vertices and surviving vertices (those reached by fake news before correct news) in generation $k$ by $Z_k$ and $Z^{\sss \mathcal{F}}_k$, respectively.

\begin{definition}[Weak survival]
\label{def-weak-surv-tree}
{\rm We say that the fake news \emph{survives in the weak sense} when}
	\eqn{
	\liminf_{K\rightarrow \infty}\liminf_{k\rightarrow \infty}\prob(Z^{\sss \mathcal{F}}_k\geq K)>0.
	}
\hfill$\ensymboldefinition$
\end{definition}

\begin{definition}[Strong survival]
\label{def-strong-surv-tree}
{\rm We say that the fake news {\em survives in the strong sense} when}
	\eqn{
	\liminf_{\eta\searrow 0}\liminf_{k\rightarrow \infty}\prob\left(\frac{Z^{\sss \mathcal{F}}_k}{Z_k}\geq \eta \right)>0.
	}
\hfill$\ensymboldefinition$
\end{definition}

\begin{definition}[Intermediate survival]
\label{def-intermediate-surv-tree}
{\rm We say that the fake news lower {\em survives in the $g_k$-intermediate sense} for some $g_k\rightarrow 0$ when there exists a constant $C>1$ such that
	\eqn{
	\liminf_{k\rightarrow \infty}\prob\left(\frac{Z^{\sss \mathcal{F}}_k}{Z_k} \geq g_k/C\right)>0.
	}
We say that the fake news upper {\em dies in the $g_k$-intermediate sense} for some $g_k\rightarrow 0$ when $(Z^{\sss \mathcal{F}}_k/g_k Z_k)_{k\geq 1}$ is a tight sequence of random variables.
We say that the fake news {\em survives in the $g_k$-intermediate sense} if it lower survives and upper dies in the $g_k$-intermediate sense.}\hfill$\ensymboldefinition$
\end{definition}

For every tree vertex $w=\root i_1\cdots i_l$ we denote by $L^{\sss \mathcal{R}}_w$ and $L^{\sss \mathcal{F}}_w$, respectively, the times it takes for fake and correct news, respectively, to reach this vertex from its parent. We assume that random variables $(L^{\sss \mathcal{R}}_w -L^{\sss \mathcal{F}}_w)_w$ are i.i.d.\ and satisfy that $\E(L^{\sss \mathcal{R}} - L^{\sss \mathcal{F}}) < \infty$, where $L^{\sss \mathcal{R}}$ and $L^{\sss \mathcal{F}}$ are typical random variables. We emphasize again that we do not assume any dependence structure between $L^{\sss \mathcal{R}}_w$ and $L^{\sss \mathcal{F}}_w$ for the same tree vertex $w$.

This section is organized as follows. In Section \ref{sec-BRW}, we reduce the problem of fake news on a branching process tree to a branching random walk problem, which we then use to determine that strong survival can only occur when $\E(L^{\sss \mathcal{R}}) > \E(L^{\sss \mathcal{F}})$. In Section \ref{sec-weak-tree}, we use this reduction to investigate when weak survival occurs on a branching process tree.

\subsection{Reduction to a branching random walk and strong survival}
\label{sec-BRW}
For each tree vertex $w=\root i_1 \cdots i_l$, let $S_w= \sum_{j=1}^l \big(L^{\sss \mathcal{R}}_{\root i_1 \cdots i_j} - L^{\sss \mathcal{F}}_{\root i_1 \cdots i_j}\big)$, and note that fake news reaches a tree vertex $w=\root i_1 \cdots i_l$ as long as $S_j > -d$ for all $j\in[l]$. We can thus consider our competition model as a ``branching random walk with killing": i.e., the spread of the fake news gets stopped at a vertex $w$ for which the corresponding sum $S_w$ is smaller than $-d$. We can however also simply let both epidemics evolve up to a generation $k$, say, and then look at all the branches and check whether or not there is one where the sums were always above $-d$ for the process with delay, or $S_j > 0$ for all $j\in[l]$ for the process with no delay.
 

Let us define
	$$
	S_0 = 0, \qquad S_k = \sum_{i=1}^k X_i, \qquad n\ge 1,
	$$
where $X_i$ are i.i.d.\ with the distribution of $L^{\sss \mathcal{R}}-L^{\sss \mathcal{F}}$. Thus, $(S_k)_{k \ge 0}$ is a generic random walk representing the delays between correct and fake news on a generic branch of a tree. Let us also define, for any $d \ge 0$,
	$$
	\tau_d = \inf\{k \ge 1\colon S_k \leq -d\},
	$$
with the convention that $\tau_d = \infty$ if $S_k > -d$ for all $k \ge 1$. If the initial delay between fake and correct news is $d$, then $\tau_d$ is the first generation where correct news arrived before fake news along a branching process branch, thus killing fake news on this branch. Now clearly
	\eqn{
	\label{rewrite-indicators}
	Z_k^{\sss \mathcal{F}} = \sum_{i=1}^{Z_k} \indicwo{i},
	}
where every $\indicwo{i}$ is the indicator that vertex $i$ is a surviving vertex, and $\P(\indicwo{i}=1) = \P(\tau_d > k)$. The above immediately implies that
	$$
	\E(Z_k^{\sss \mathcal{F}}\mid Z_k) = Z_k \P(\tau_d > k),
	$$
and thus
	\begin{equation} \label{eq:exp_1}
	\E(Z_k^{\sss \mathcal{F}}) = \E(Z_k) \P(\tau_d > k) = \nu^k \P(\tau_d > k),
	\end{equation}
and also
	\begin{equation} \label{eq:exp_2}
	\E \left(\frac{Z_k^{\sss \mathcal{F}}}{Z_k}\right) = \E\left( \E\Big(\frac{Z_k^{\sss \mathcal{F}}}{Z_k}\mid Z_k\Big)\right) = \P(\tau_d > k).
	\end{equation}

The above results demonstrate that in terms of expectations we can reduce everything to the first-passage time for a random walk, which is relatively well understood. We proceed to exploit these to prove exact results on strong and weak survival. Note, however, that in all cases where asymptotics for $\P(\tau_d > k)$ are known, the above also gives an indication of when intermediate survival holds (see Theorem \ref{thm-inter-tree}). We refer to \cite{DS2013} and references therein for a collection of results on the exact asymptotics for the tail of $\tau_d$.



The following result characterizes strong survival on trees:

\begin{theorem}[Strong survival on a tree]
\label{thm:strong_tree}
Condition on survival. Fake news strongly survives if and only if $\E(L^{\sss \mathcal{R}}) > \E(L^{\sss \mathcal{F}})$.
\end{theorem}

\proof If $\E(L^{\sss \mathcal{R}}) \le \E(L^{\sss \mathcal{F}})$, then $\P(\tau_d > k) \to 0$ as $k \to \infty$, and strong survival is thus impossible, thanks to the Markov inequality (first-moment method).

Assume now that $\E(L^{\sss \mathcal{R}}) > \E(L^{\sss \mathcal{F}})$. In this case \eqref{eq:exp_2} implies that $\E \left(Z_k^{\sss \mathcal{F}}/Z_k\right) = \P(\tau_d > k)\to p_d > 0$ as $k \to \infty$. As it is clear that $\E \Big(\big(Z_k^{\sss \mathcal{F}}/Z_k\big)^2\Big) \le 1$, the Paley-Zigmund inequality (second-moment method) finishes the proof. 
\qed


\subsection{Weak survival on a tree}
\label{sec-weak-tree}
Since fake news strongly survives when $\E(L^{\sss \mathcal{R}}) > \E(L^{\sss \mathcal{F}})$, we focus on the case $\E(L^{\sss \mathcal{R}}) \le \E(L^{\sss \mathcal{F}})$ here.

Recall the definition
	\begin{equation}
	\psi(s) = \E[\e^{s(L^{\sss \mathcal{R}} - L^{\sss \mathcal{F}})}]. 
	\end{equation}

\begin{theorem}[Weak survival on a tree]
\label{thm:main_tree}
Condition on survival.
\begin{itemize} 
\item[(i)] Assume, in addition, $\psi(s) < \infty$ for some $s>0$ and there exists $0 < h < s$ such that $\psi'(h) = 0$. Denote  $\rho = \psi(h)$. If $\rho > 1/\nu$, then fake news survives in the weak sense. If $\rho \leq 1/\nu$, fake news does not survive in the weak sense.
\item[(ii)] If $\psi(s) = \infty$ for all $s>0$, then fake news survives in the weak sense.
\end{itemize}
\end{theorem}
\medskip

Note that as we do not assume that $\nu < \infty$, statement (i) means that if $\nu = \infty$, then fake news weakly survives, regardless of the value of $\rho$.

Theorem \ref{thm:main_tree} for branching process trees is quite similar to Theorem \ref{thm-weak-CM} for the configuration model. In particular, the proof of Theorem \ref{thm-weak-CM} in Section \ref{sec-weak-CM} will be a combination of local convergence together with the analysis on the tree in Theorem \ref{thm:main_tree}. We note also that, as in the case of Theorem \ref{thm-weak-CM}, Theorem \ref{thm:main_tree} covers all possible distributions, apart from some pathological ones mentioned in Remark \ref{rem-pathology}.

\proof For $\nu < \infty$, part (i) is equivalent to results of \cite{BigLubShwWei91} (see also \cite{AidJaf11} for a comment on the case $\rho=1/\nu$). Earlier work includes that of Kingman \cite{King75} and Biggins \cite{Bigg76}. We note that these results are almost immediate from \eqref{eq:exp_1}, since in the light-tailed case, as $k\rightarrow \infty$,
	\eqn{
	\P(\tau_d > k)= C (1+o(1))k^{-3/2} \rho^{-k},
	}
for some constant $C$.

For $\nu = \infty$, note that if fake news weakly survives on a tree, then it will also weakly survive on a strictly larger tree such that the number of children of each vertex is w.p. $1$ not smaller than that in the original tree. Let, as above, $C_w$ denote the number of children of a vertex $w$. As $\E(C_w) = \infty$, we can find $M$ such that $\E(\min\{C_w,M\}) > 1/\rho$. Fake news then survives on a tree with every vertex $w$ having $\min\{C_w,M\}$ children and hence also on the original tree.

Let us now prove (ii). Denote by $\mu = - \E(L^{\sss \mathcal{R}}  - L^{\sss \mathcal{F}}) \geq 0$ and choose a positive $\delta$. Using ideas from, e.g. \cite[Proposition 3.1]{Zwart01}, we can write
	\begin{align*}
	\P(\tau_d > k) & \ge \P(X_1 > k(\mu + \delta), \min\{S_2,\ldots,S_k\} > 0)\\
	& \ge \P(X_1 > k(\mu + \delta)) \P(\min\{\xi_2, \xi_2+\xi_3,\ldots,\xi_2+\xi_3+\cdots+\xi_k\} > - k(\mu + \delta)) \\
	& \ge (1-\delta) \P(X_1 > k(\mu + \delta))
	\end{align*}
for large enough $k$, due to the strong law of large numbers. As $\psi(s) = \infty$ for all $s>0$
	$$
	\e^{-\varepsilon k} = o(\P(X_1  > k))
	$$
for any $\varepsilon > 0$. Choose $K$ large enough so that
	$$
	\P(L^{\sss \mathcal{R}}  > k) \ge C \e^{-\varepsilon k},
	$$
for a constant $C$ and for all $k \ge K$. Then, assuming $\nu < \infty$, and using \eqref{eq:exp_2},
	$$
	\E(Z_k^{\sss \mathcal{F}}) \ge C (1-\delta) \nu^k \e^{-\varepsilon (\mu+\delta) k} \to \infty
	$$
exponentially fast, as long as $\varepsilon < \frac{\log \nu}{\mu+\delta}$. We can then use the arguments from \cite{King75} and \cite{Bigg76} to finish the proof. It is also easy to use argument similar to those applied in (i) to finish the proof for $\nu = \infty$. \qed

We continue by discussing some interesting examples, and relate our results to the literature:

\begin{remark}[Survival when traversal times have same law]
\label{remark:identical}
{\rm An interesting example is when both news spread independently and according to the same law, i.e., $L^{\sss \mathcal{F}}$ and $L^{\sss \mathcal{R}}$ are independent over all edges and have the same distribution. Denote by $L$ a random variable with that distribution and denote by $\varphi(s)$ the generating function of $L$. In this case,
	$$
	\rho = \inf_{s \ge 0} \psi(s) = \inf_{s \ge 0} \varphi(s) \varphi(-s) = \inf_{s \ge 0} \E[\e^{sL}] \E\left[\frac{1}{\e^{sL}}\right] > 1,
	$$
due to strict convexity of the function $1/x$ and Jensen's inequality. Therefore, $\rho > 1/\nu$ and we have weak survival of the fake news with a positive probability for all light-tailed laws of the transmission times, as long as these laws are the same for the fake and the correct news. This justifies the title ``it is hard to kill fake news''. It is in this case impossible to kill fake news, unless it is done in the first few steps.}\hfill$\ensymboldefinition$
\end{remark}

\begin{remark}[Exponential distributions]
\label{rem-exp}
{\rm Another interesting case is when $L^{\sss \mathcal{R}}$ and $L^{\sss \mathcal{F}}$ are independent and both have an exponential distribution. Let us denote their parameters by $\mu^{\sss \mathcal{R}}$ and $\mu^{\sss \mathcal{F}}$ and assume that $\mu^{\sss \mathcal{R}} > \mu^{\sss \mathcal{F}}$. A simple calculation shows that $\rho = 4 \mu^{\sss \mathcal{R}} \mu^{\sss \mathcal{F}}/(\mu^{\sss \mathcal{R}} + \mu^{\sss \mathcal{F}})^2$, and survival will depend on whether or not this is larger than $1/\nu$.
Note that this recovers the results of \cite{Kord05}.}\hfill$\ensymboldefinition$
\end{remark}

\subsection{Intermediate survival on a tree}

One can generalize results of Theorem \ref{thm:main_tree} and prove results on intermediate survival:

\begin{theorem}[Intermediate survival on a tree]
\label{thm-inter-tree}
Condition on survival and assume that $\E[L^{\sss \mathcal{R}}] \le \E[L^{\sss \mathcal{F}}]$. Define
$$
p_k = \P(\tau_d > k)
$$
with $\tau_d$ defined above. Then fake news upper dies in the $g_k$-intermediate sense for any sequence $g_k$ such that $p_k = o(g_k)$.
\end{theorem}

\proof The statement follows directly from \eqref{eq:exp_2} and the first-moment method.
\qed

\section{Proofs for the configuration model}
\label{sec-CM}
In the graph setting we will assume that both epidemics start from the same, randomly chosen, vertex, say, $U$. We again assume that the correct news starts at the same time as the fake news, or that there is a delay, so that the fake news starts earlier. We know that the CM is approximated well by a branching process tree, so that we can prove weak survival, which is a local property, using local convergence techniques. For strong survival, it is not sufficient to consider local convergence, as strong survival is a global notion. However, we also know that the percolation time to another randomly chosen vertex is of order $\log{n}/\lambda^{\sss \mathcal{R}}$ and $\log{n}/\lambda^{\sss \mathcal{F}}$, respectively, for the correct and fake news. We are going to exploit this in our proofs.

This section is organized as follows. We start in Section \ref{sec-weak-CM} by using local convergence arguments and Theorem \ref{thm:main_tree} to prove Theorem  \ref{thm-weak-CM}. In Section \ref{sec-strong-CM}, we prove the strong survival part of Theorem \ref{thm-strong-CM}, which is the the most involved part of our proof. In Section \ref{sec-no-strong-CM}, we prove the no strong survival part of Theorem \ref{thm-strong-CM}. We close in Section \ref{sec-strong-CM-explos} by proving the strong survival for infinite-variance degrees in Theorem \ref{thm-strong-CM-explos}. We focus on the case without delay (recall Definition \ref{def-start-delay}).

\subsection{Weak survival on CM: Proof of Theorem \ref{thm-weak-CM}}
\label{sec-weak-CM}
We use the notion of local convergence (LC). See \cite[Chapter 2]{Hofs18} for an extensive explanation of local convergence, and  \cite[Chapter 4]{Hofs18} for a proof of local convergence in probability for the configuration model. By LC, 
	\eqn{
	\lim_{n\rightarrow \infty}\prob(N^{\sss \mathcal{F}}_n\geq K)=\prob(N^{\sss \mathcal{F}}\geq K),
	}
where $N^{\sss \mathcal{F}}$ is the number of vertices that hear the fake news on the unimodular branching process tree with distribution $(p_k)_{k\geq 1}$, with $p_k=\prob(D=k)$.
In such a branching process, the root has offspring $(p_k)_{k\geq 1}$, while any other vertex has offspring distribution $(p_k^\star)_{k\geq 1}$, given by $p_k^\star=\prob(D^\star-1=k)=(k+1)p_{k+1}/\sum_{l\geq 1} lp_l$.
Under the conditions for weak survival of Theorem \ref{thm-weak-CM}, we have that, in the case (i) by Theorem \ref{thm:strong_tree} and in cases (ii) and (iii) by Theorem \ref{thm:main_tree},
	\eqn{
	\liminf_{K\rightarrow \infty}\prob(N^{\sss \mathcal{F}}\geq K)>0,
	}
as required. On the other hand, when $\varphi^{\sss \mathcal{R}}(s) < \infty$ for some $s>0$ and $\rho \leq 1/\nu$, again by Theorem \ref{thm:main_tree},
the above limit is 0, so that the fake news does not weakly survive. This completes the proof of Theorem \ref{thm-weak-CM}.
\qed

\subsection{Proof of strong survival in Theorem \ref{thm-strong-CM} when $\lambda^{\sss \mathcal{F}}>\lambda^{\sss \mathcal{R}}$}
\label{sec-strong-CM}
In this section, we prove the strong survival statement in Theorem \ref{thm-strong-CM} when $\lambda^{\sss \mathcal{F}}>\lambda^{\sss \mathcal{R}}$. Recall that the number of individuals that are exposed to the fake news is denoted by $N_n^{\sss \mathcal{F}}$. The key result in the current proof is the following proposition:

\begin{proposition}[Positive proportion is exposed to fake news]
\label{prop-exposed-fake-news}
Under the assumptions of Theorem \ref{thm-strong-CM}, when $\lambda^{\sss \mathcal{F}}>\lambda^{\sss \mathcal{R}}$,
	\eqn{
	\label{Nn-expec-LB}
	\liminf_{n\rightarrow \infty} \frac{1}{n}\expec[N_n^{\sss \mathcal{F}}]>0.
	}
\end{proposition} 
\medskip

Before proving Proposition \ref{prop-exposed-fake-news}, we use it to complete the proof of the strong survival claim in Theorem \ref{thm-strong-CM}:
\medskip

\noindent
{\it Proof of the strong survival claim in Theorem \ref{thm-strong-CM} subject to Proposition \ref{prop-exposed-fake-news}.} Recall that by Definition \ref{def-strong-surv}, we need to prove that there exists $\eta>0$ such that
	\eqn{
	\label{cond-strong-surv}
	\liminf_{n\rightarrow \infty}\prob(N^{\sss \mathcal{F}}_n\geq \eta n)>0.
	}
We argue by contradiction. Assume that \eqref{cond-strong-surv} fails. Then, $N^{\sss \mathcal{F}}_{n_k}/{n_k}\convp 0$ along a subsequence $(n_k)_{k\geq 1}$ such that $n_k\rightarrow \infty$.  
However, by dominated convergence and the fact that $N^{\sss \mathcal{F}}_n/n\leq 1$, it then also follows that $\frac{1}{n_k}\expec[N_{n_k}^{\sss \mathcal{F}}]\rightarrow 0$, which contradicts \eqref{Nn-expec-LB} in Proposition \ref{prop-exposed-fake-news}. We conclude that \eqref{cond-strong-surv} holds, which completes the proof of the strong survival claim in Theorem \ref{thm-strong-CM} subject to \eqref{Nn-expec-LB}.
\qed
\medskip

In the remainder of this section, we complete the proof of Proposition \ref{prop-exposed-fake-news}.

\begin{proof}[Proof of Proposition \ref{prop-exposed-fake-news}] We first write
	\eqn{
	\frac{1}{n}\expec[N_n^{\sss \mathcal{F}}]=\prob(\Ver' \text{ exposed to fake news from source }\Ver),
	}
where we take the expectation w.r.t.\ the random source $\Ver$ of the fake news epidemic, and the random destination $\Ver'$. Our aim is thus to prove that
	\eqn{
	\label{aim-prop-exposed-fake-news}
	\liminf_{n\rightarrow \infty} \prob(\Ver' \text{ exposed to fake news from source }\Ver)>0.
	}
It is here that we crucially rely on the results on first-passage percolation on $\CMnd$ derived in 	\cite{BhaHofHoo17}, under the conditions that the degree distribution in $\CMnd$ satisfies Conditions \ref{cond-degrees-regcond}(a)-(c), and that \eqref{cond-D2logD} holds.
\medskip

As in \cite{BhaHofHoo17}, take $B>0$ large, and we perform the unrestricted fake and correct news epidemics to time $1/(2\lambda_n^{\sss \mathcal{F}})\log{n} +B$ from the uniform source vertex $\Ver$. We do this in a joint way, but we {\em ignore the competition}. Thus, the fake news keeps on spreading from vertices that have also been exposed to the correct news. Hence, in this setting, the set of vertices exposed to the fake news is a subset of the vertices that are actually reached by the fake news in which vertices stop spreading the fake news upon being exposed to the correct news. We therefore call the set of vertices reached by the unrestricted fake news epidemic the {\em potential} fake news epidemic. Note that $\Ver'$ {\em will} eventually be reached by the unrestricted fake news epidemic as soon as $\Ver$ and $\Ver'$ are in the same connected component, which occurs with strictly positive probability under the assumptions of Theorem \ref{thm-strong-CM}. We then grow the {\em backwards} potential fake news epidemic up to time $1/(2\lambda_n^{\sss \mathcal{F}})\log{n} +B$ from the uniform destination vertex $\Ver'$. It is proved in \cite[Theorem 1.2]{BhaHofHoo17} (see also \cite[Proposition 3.2]{BhaHofHoo17}) that, conditionally on $\Ver, \Ver'$ being in the same connected component in $\CMnd$, the two potential fake new epidemics meet with a probability that is arbitrarily close to 1 when $B$ grows large.

The remainder of the proof is to show that $\Ver'$ will with positive probability also be reached by the fake news when $\lambda^{\sss \mathcal{F}}>\lambda^{\sss \mathcal{R}}$. Therefore, we need to argue that, with positive probability, the correct news does not {\em block} the fake news epidemic on its way to $\Ver'$, i.e.,
	\eqn{
	\label{aim-prop-exposed-fake-news-2}
	\liminf_{n\rightarrow \infty} \prob(\text{fake news epidemic from }\Ver \text{ to }\Ver' \text{ is not blocked})>0.
	} 
At the same time, we also grow the correct news epidemic from $\Ver$, and note that this is simply regular first passage percolation, so that the results from \cite{BhaHofHoo17} also apply to it. However, the fact that the correct and fake news traversal times across edges are not necessarily independent makes the dependencies between the various epidemics subtle. We start by setting the scene, and start by collecting some crucial results from \cite{BhaHofHoo17}.
\medskip

By \cite{BhaHofHoo17}, we observe the following:

\begin{itemize}
\item[$\rhd$] The size of the potential fake news epidemic from $\Ver$ grows like $\Theta_{\sss \prob}(1) \e^{B\lambda_n^{\sss \mathcal{F}}}\sqrt{n},$ by \cite[Proposition 2.3(a)]{BhaHofHoo17};

\item[$\rhd$] The size of the potential fake news epidemic from $\Ver'$ grows  like $\Theta_{\sss \prob}(1) \e^{B\lambda_n^{\sss \mathcal{F}}}\sqrt{n},$ again by \cite[Proposition 2.3(a)]{BhaHofHoo17};

\item[$\rhd$] With probability that is close to 1 for large $B$, the smallest-weight path of the potential fake news is formed between the two smallest-weight trees from $\Ver$ and $\Ver'$, by \cite[Proposition 2.3(a), Theorem 3.1 and Proposition 3.2]{BhaHofHoo17};

\item[$\rhd$] The correct news reaches at most $\Theta_{\sss \prob}(1) \e^{B\lambda^{\sss \mathcal{R}}} n^{\lambda_n^{\sss \mathcal{R}}/(2\lambda_n^{\sss \mathcal{F}})}\ll \sqrt{n}$ vertices, and its smallest-weight graph is whp a tree, again by \cite[Proposition 2.3(a)]{BhaHofHoo17}.



\end{itemize}
The combined smallest-weight graphs of fake and correct news can be thus seen as a graph of joint paths, appended by the paths to vertices that are found by the correct news (which jointly form a tree), but not by the fake news, as well as the paths to vertices found by the fake news, but not by the correct news. This last group of paths is not necessarily a tree, since its size is of order $\sqrt{n}$, which is such that the first deviations from a tree start to emerge. While the structure of the subgraph of all vertices found by the fake news is not necessarily a tree, we do know that there are $\Theta_{\sss \prob}(1) \e^{B\lambda^{\sss \mathcal{F}}}\sqrt{n}$ partially explored edges on the boundaries of the smallest-weight graphs from both $\Ver$ and $\Ver'$. 
\smallskip

By the second item above, the two smallest-weight potential fake news graphs from $\Ver$ and $\Ver'$ are connected to form the smallest-weight path between $\Ver$ and $\Ver'$. This occurs through a so-called {\em collision edge} (see \cite[Section 2.2, and the main result of their occurrence in Proposition 2.3 and Theorem 3.1]{BhaHofHoo17}). There are many possible collision edges, and only one of them gives rise to the smallest-weight path (it is one of the first ones, see \cite[Proposition 3.2]{BhaHofHoo17}). Denote this edge by $e=(x,y)$, where $x$ and $y$ are the two half-edges that are paired together to form $e$, and, upon its pairing, $x$ was on the boundary of the smallest-weight graph of $\Ver$, while $y$ was on the boundary of the smallest-weight graph of $\Ver'$.
\smallskip

By the above construction, we note that if the potential fake news is blocked by the correct news, so that it does not reach $\Ver'$, then one of the vertices along the smallest-weight potential fake news paths must be exposed to the correct news, before being able to transmit the fake news. Such a blocking can occur when the correct news overtakes the smallest-weight path of the potential fake news, which we call a {\em path-blocking}. It can also occur through a different path, which we call a {\em loop-blocking}. We divide such a loop-blocking into two possible cases. Indeed, a loop-blocking can occur at one of the vertices on the smallest-weight path from $y$ to $\Ver'$, which we call a {\em late loop-blocking}, or at one of the vertices on the smallest-weight path from $\Ver$ to $x$, which we call an {\em  early path-blocking}. In summary, the following three blockings can arise:

\begin{itemize}
\item[(a)] The correct news overtakes the fake news along the smallest-weight path of the potential fake news epidemic from $\Ver$ to the half-edge along its boundary that is first reached by the potential fake news epidemic from $\Ver'$, which is a path-blocking;

\item[(b)] The potential fake news epidemic from $\Ver$ is overtaken by the correct news along the path from $\Ver$ to $x$, which is an early loop-blocking; or

\item[(c)] The potential fake news epidemic from $y$ to $\Ver'$ is overtaken by the correct news, which is a late loop-blocking.

\end{itemize}
We will show that the two loop-blockings occur with vanishing probability, while the path-blocking does not occur with strictly positive probability, which proves \eqref{aim-prop-exposed-fake-news-2}. We continue to investigate all these blockings, starting with the loop-blockings.
\smallskip

\paragraph{\bf Early loop-blocking with vanishing probability.} 
Recall that we construct the potential fake news smallest-weight graph {\em before} creating the correct news smallest-weight graph from $\Ver$. Thus, for the correct news to overtake the potential fake news epidemic between $\Ver$ and $x$, upon traversing one of the alive half-edges incident to the correct news smallest-weight graph, it must pair with one of the half-edges incident to the potential fake news smallest-weight graph from $\Ver$ at time
$1/(2\lambda_n^{\sss \mathcal{F}})\log{n} +B$. Recall that the number of half-edges incident to the potential fake news smallest-weight graph from $\Ver$ grows like $\Theta_{\sss \prob}(1) \e^{B\lambda^{\sss \mathcal{F}}}\sqrt{n}.$ Since half-edges are paired uniformly at random and there are still $\ell_n(1+o(1))$ unpaired half-edges, each time a half-edge is traversed by the correct news, it has probability at most
	\eqn{
	\frac{1}{\ell_n(1+o(1))}\Theta_{\sss \prob}(1) \e^{B\lambda^{\sss \mathcal{F}}}\sqrt{n}=\Theta_{\sss \prob}(n^{-1/2})
	}
of being paired to a half-edge incident to the potential fake news smallest-weight graph from $\Ver$ at time $1/(2\lambda_n^{\sss \mathcal{F}})\log{n} +B$. Since the correct news smallest-weight graph has size $\op(\sqrt{n})$, by the union bound, this is unlikely to happen. We conclude that the probability of an early loop-blocking vanishes, as required.
\smallskip

\paragraph{\bf Late loop-blocking with vanishing probability.} 
This proof is very similar to the previous proof. Recall that we construct the potential fake news smallest-weight graph {\em before} creating the correct news smallest-weight graph from $\Ver'$ backwards in time. Thus, for the correct news to overtake the potential fake news epidemic between $y$ and $\Ver'$, upon traversing one of the alive half-edges incident to correct news smallest-weight graph, it must pair with one of the half-edges incident to the potential fake news smallest-weight graph from $\Ver'$ at time$1/(2\lambda_n^{\sss \mathcal{F}})\log{n} +B$. The remainder of the proof can be completed as before, and we again conclude that the probability of a late loop-blocking vanishes, as required.
\smallskip

\paragraph{\bf No path-blocking with positive probability.} 
We finally show that the path-blocking event does not occur with strictly positive probability. This is the most involved and innovative part of the argument, and only here do we need to investigate the possible {\em dependence} between the fake and correct news traversal times. Recall that we do not assume that these traversal times are independent, and we only assume that $\lambda^{\sss \mathcal{F}}>\lambda^{\sss \mathcal{R}}$, which is a weak assumption on marginal distributions. 

We start by investigating the traversal times of the fake news, as well as the correct news to the half-edge that connects the two potential fake news smallest-weight graphs. 
\smallskip

Introduce the fake news {\em stable-age distribution} via
	\eqn{
	\label{eq:stable_fake}
	\overline{F}^{\sss{\mathcal F}}(x)=\prob(\bar{L}^{\sss{\mathcal F}}\leq x)= \nu \int_0^x \e^{-\lambda^{\sss{\mathcal F}} y} F^{\sss{\mathcal F}}(dy),
	}
and its mean by
	\eqn{
	\label{eq:stable_fake_mean}
	\overline{\nu}^{\sss{\mathcal F}} = \nu \int_0^\infty x \e^{-\lambda^{\sss{\mathcal F}} x} dF^{\sss{\mathcal F}}(x) = - \nu (\varphi^{\sss{\mathcal F}})'(-\lambda^{\sss{\mathcal F}})
	}
(recall the notation $\varphi^{\sss{\mathcal F}}(s) = \E\left[\e^{s L^{\sss{\mathcal F}}}\right]$ introduced after Theorem \ref{thm-strong-CM}). Further, we let $\bar{L}^{\sss{\mathcal R}}$ denote the distribution of $L^{\sss{\mathcal R}}$ under the stable age distribution $\overline{F}^{\sss{\mathcal F}}$ as
	\eqn{
	\label{eq:stable_real}
	\prob(\bar{L}^{\sss{\mathcal R}}\leq x)= \nu \int_0^\infty \e^{-\lambda^{\sss{\mathcal F}} y} F^{\sss{\mathcal{R}, \mathcal{F}}}(x, dy),
	}
where $F^{\sss{\mathcal{R}, \mathcal{F}}}$ is the joint cumulative distribution function of $(L^{\sss{\mathcal R}},L^{\sss{\mathcal F}})$. In particular, $\bar{L}^{\sss{\mathcal R}}$ has the same law as $L^{\sss{\mathcal R}}$ when $L^{\sss{\mathcal R}}$ and $L^{\sss{\mathcal F}}$ are independent. 
\smallskip

These random variables are relevant, since the edge-weights along the potential fake news smallest-weight path are close to i.i.d.\ random variables with the stable-age distribution, while the corresponding correct news traversal times are then distributed as $\bar{L}^{\sss{\mathcal R}}$. Thus, when a path-blocking occurs, the sum of the $\bar{L}^{\sss{\mathcal R}}$ correct news edge-weights along the path needs to be smaller than the sum of the $\bar{L}^{\sss{\mathcal F}}$ stable-age distributed potential fake-news traversal times. In the following lemma, we show that $\lambda^{\sss \mathcal{F}}>\lambda^{\sss \mathcal{R}}$ implies that the mean of $\bar{L}^{\sss{\mathcal R}}$ is larger than that of $\bar{L}^{\sss{\mathcal F}}$:

\begin{lemma}[Relation stable-age distribution fake news and the expected correct news]
\label{lem-sad-correct}
Assume that $\lambda^{\sss \mathcal{F}}>\lambda^{\sss \mathcal{R}}$. Then, 
	\eqn{
	\label{nu-bar-LbarR}
	\overline{\nu}^{\sss{\mathcal F}} < \E[\bar{L}^{\sss{\mathcal R}}].
	}
\end{lemma}


\begin{proof}
%
Using the definitions of $\overline{\nu}^{\sss{\mathcal F}}$ and $\bar{L}^{\sss{\mathcal R}}$, \eqref{nu-bar-LbarR} may be rewritten as
	$$
	\E[L^{\sss{\mathcal F}} \e^{-\lambda^{\sss{\mathcal F}} L^{\sss{\mathcal F}}}] 
	< \E[L^{\sss{\mathcal R}} \e^{-\lambda^{\sss{\mathcal F}} L^{\sss{\mathcal F}}}].
	$$
To prove the above, we can now use the inequality 
	$$
	\e^{-\lambda^{\sss{\mathcal F}} x} - \e^{-\lambda^{\sss{\mathcal F}} y} 
	=\e^{-\lambda^{\sss{\mathcal F}} x}(1-\e^{-\lambda^{\sss{\mathcal F}} (y-x)}) 
	\le \lambda^{\sss{\mathcal F}} \e^{-\lambda^{\sss{\mathcal F}} x} (y-x).
	$$
which, since $1-\e^{-x}\leq x$, is valid for all $x$ and $y$. In turn, this implies
	\eqn{
	\lambda^{\sss{\mathcal F}} \E[(L^{\sss{\mathcal R}} - L^{\sss{\mathcal F}}) \e^{-\lambda^{\sss{\mathcal F}} L^{\sss{\mathcal F}}}] 
	\ge \E[\e^{-\lambda^{\sss{\mathcal F}} L^{\sss{\mathcal F}}}] - \E[\e^{-\lambda^{\sss{\mathcal F}} L^{\sss{\mathcal R}}}] 
	= \frac{1}{\nu} - \E[\e^{-\lambda^{\sss{\mathcal F}} L^{\sss{\mathcal R}}}] > 0,
	}
where we use that $\E[\e^{-\lambda^{\sss{\mathcal F}} L^{\sss{\mathcal R}}}]  < \E[\e^{-\lambda^{\sss{\mathcal R}} L^{\sss{\mathcal R}}}] = 1/\nu$ because $\lambda^{\sss{\mathcal F}} > \lambda^{\sss{\mathcal R}}$.
\end{proof}
\medskip


Before discussing corollaries of Lemma \ref{lem-sad-correct}, let us describe how it helps proving the result needed in Remark \ref{rem-weak-vs-strong}:

\begin{remark}[Weak survival is really weaker than strong survival]
\label{rem-weak-vs-strong-1}
{\rm Assume that $\lambda^{\sss{\mathcal F}} > \lambda^{\sss{\mathcal R}}$. Then
\begin{align*}
\rho & = \inf_{s \ge 0} \E\big[\e^{s(L^{\sss{\mathcal R}} - L^{\sss{\mathcal F}})}\big] = \frac{1}{\nu} \inf_{s \ge 0} \E\left[\e^{s(\overline{L}^{\sss{\mathcal R}} - \overline{L}^{\sss{\mathcal F}})} \e^{\lambda^{\sss{\mathcal F}}\overline{L}^{\sss{\mathcal F}}} \right] > \frac{1}{\nu} \inf_{s \ge 0} \E\big[\e^{s(\overline{L}^{\sss{\mathcal R}} - \overline{L}^{\sss{\mathcal F}})}\big] = \frac{1}{\nu},
\end{align*}
where the last equality is due to the fact that the function $s \mapsto \E\big[\e^{s(\overline{L}^{\sss{\mathcal R}} - \overline{L}^{\sss{\mathcal F}})}\big]$ for $s \ge 0$ takes its infimum at $s=0$, since its derivative satisfies
$$
\E\left[(\overline{L}^{\sss{\mathcal R}} - \overline{L}^{\sss{\mathcal F}})\e^{s(\overline{L}^{\sss{\mathcal R}} - \overline{L}^{\sss{\mathcal F}})}\right] \ge \E\big[(\overline{L}^{\sss{\mathcal R}} - \overline{L}^{\sss{\mathcal F}})\big] \E\big[\e^{s(\overline{L}^{\sss{\mathcal R}} - \overline{L}^{\sss{\mathcal F}})}\big] \ge 0,
$$
thanks to Lemma \ref{lem-sad-correct}.}\hfill$\ensymboldefinition$
\end{remark}

We next investigate the implications of Lemma \ref{lem-sad-correct}. Consider the continuous-time branching process where every individual in its lifetime gives birth to a random number of children taken independently over individuals from a distribution $D^\star-1$ with a finite expectation $\nu=\expec[D^\star-1]> 1$. The children of individuals are born at random times after the birth of the individual, taken from distribution $F^{\sss{\mathcal F}}$, and these birth times are independent across individuals and independent from the numbers of children. This is a so-called age-dependent branching process. An individual dies immediately after giving birth to its last child. Denote the time of birth of an individual $v$ in generation $k$ by
	\eqn{
	S_k^{v, \sss{\mathcal F}} = \sum_{i=1}^k L_i^{v,\sss{\mathcal F}}
	}
(with obvious notation for the sequence $L_i^{v,\sss{\mathcal F}}$). Also, denote the potential time of birth of the individual in the same epidemic when the birth-time distributions were replaced by $F^{\sss{\mathcal R}}$ by
	\eqn{
	S_k^{v, \sss{\mathcal R}} = \sum_{i=1}^k L_i^{v,\sss{\mathcal R}}
	}
As in our setting, we do assume independence between realisations of $(L^{\sss{\mathcal F}}, L^{\sss{\mathcal R}})$ over different individuals, but do not assume any specific dependence structure between the values of  $L^{\sss{\mathcal F}}$ and $L^{\sss{\mathcal R}}$ for any individual.

For a random walk $(S_k^{\sss{\mathcal R}} - S_k^{\sss{\mathcal F}})_{k \ge 0}$, denote
$$
\sigma = \inf\{k \ge 0\colon S_k^{\sss{\mathcal R}} - S_k^{\sss{\mathcal F}} < 0\},
$$
and the corresponding stopping time for the random walk $(\bar{S}_k^{\sss{\mathcal R}} - \bar{S}_k^{\sss{\mathcal F}})_{k \ge 0}$ by $\bar{\sigma}$ (this is defined in the obvious manner by replacing birth-time distributions with their stable-age counterparts). Note that, thanks to Lemma \ref{lem-sad-correct},
	\begin{equation} \label{eq:sigma_bar}
	p^\star \equiv \P(\bar{\sigma} = \infty) > 0.
	\end{equation}

Let $N^{\sss{\mathcal F}}_t$ denote the number of individuals $v$ alive at time $t$ such that $S_k^{u, \sss{\mathcal F}} < S_k^{u, \sss{\mathcal R}}$ for every $u$ which is an ancestor of $v$ (in other words, $N^{\sss{\mathcal F}}_t$ counts the number of individuals alive at time $t$ that are reached by fake news). Further, let $N^{\sss{\mathcal F}}_{t,\vep}$ be the restriction of the alive individuals counted in $N^{\sss{\mathcal F}}_t$ that have remaining lifetime at most $\vep>0$. 

Let $Z^{\sss{\mathcal F}}_{t,\vep}$ denote the number of individuals reached by the fake-news branching process, such that their life time is smaller than $\vep$, and denote
	\begin{equation} \label{eq:def_lifetime}
	H(x) = \int_0^\infty \e^{-\lambda^{\sss{\mathcal F}}z} \P(L^{\sss{\mathcal F}} \in (z,z+x)) dz, \quad x \ge 0.
	\end{equation}
The following lemma, describing the exponential growth of $Z^{\sss{\mathcal F}}_{t,\vep}$ and $N^{\sss{\mathcal F}}_{t,\vep}$, is key to our analysis:
\begin{lemma}[Survival fake news in CTBP setting]
\label{lem:branching_process_survival}
Condition on the survival of the fake-news branching process. Then
	\eqn{
	\label{remaining-lifetime}
	\lim_{t \to \infty} \e^{-\lambda^{\sss{\mathcal F}}t} \E[Z^{\sss{\mathcal F}}_{t,\vep}] = \frac{1}{\bar{\nu}} H(\varepsilon).
	}
If, in addition, $\lambda^{\sss{\mathcal F}} > \lambda^{\sss{\mathcal R}}$, then
	\eqn{
	\label{non-blocked-individuals-remaining-lifetime}
	\lim_{t \to \infty} \e^{-\lambda^{\sss{\mathcal F}}t} \E [N^{\sss{\mathcal F}}_{t,\vep}] = p^\star\frac{1}{\bar{\nu}} H(\varepsilon).
	}
\end{lemma}

Before proving the above lemma, let us note that, by taking $\varepsilon = \infty$ in \eqref{non-blocked-individuals-remaining-lifetime} above, we conclude
\eqn{
	\label{non-blocked-individuals}
	\lim_{t \to \infty} \e^{-\lambda^{\sss{\mathcal F}}t} \E [N^{\sss{\mathcal F}}_{t}] = p^\star \frac{\nu-1}{\lambda^{\sss{\mathcal F}} \bar{\nu} \nu}.
	}

Let $A_t$ be the event that a uniformly chosen alive individual in the fake-news branching process at time $t$ has in fact been reached by fake news (i.e. the fake-news epidemic was not overtaken by the correct-news one at any point on the way to this individual).

Lemma \ref{lem:branching_process_survival} implies that $\P(A_t) \to p^\star$ as $t \to \infty$. Indeed, as in \eqref{rewrite-indicators}, we can write
	$$
	N^{\sss{\mathcal F}}_{t} = \sum_{i=1}^{Z_{t}^{\sss{\mathcal F}}} \1_i,
	$$
where $Z_{t}^{\sss{\mathcal F}}$ is the number of individuals alive in the fake-news branching process at time $t$ and $\1_i$ is the indicator function that individual $i$ has the property described above. As $\1_i$ are interchangeable, it is clear that
	$$
	\expec[N^{\sss{\mathcal F}}_{t}] = \expec[Z_{t}^{\sss{\mathcal F}}] \P(A_t),
	$$
and hence $\P(A_t) \to p^\star$ as $t \to \infty$, thanks to \eqref{non-blocked-individuals} and the known result
	$$
	\lim_{t \to \infty} \e^{-\lambda^{\sss{\mathcal F}}t} \E [Z^{\sss{\mathcal F}}_{t}] =\frac{\nu-1}{\lambda^{\sss{\mathcal F}} \bar{\nu} \nu}.
	$$
Let $B_{t,\vep}$ denote the event that a (uniformly) randomly chosen individual alive in the fake-news branching process at time $t$ has remaining life time at most $\vep$. Similarly to the above, one can see from \eqref{remaining-lifetime} that 
	$$
	\P(B_{t,\vep}) \to \frac{\lambda^{\sss{\mathcal F}} \nu}{\nu-1} H(\vep).
	$$
Moreover, \eqref{non-blocked-individuals-remaining-lifetime} implies that
	$$
	\P(A_t \cap B_{t,\vep}) \to p^\star\frac{\lambda^{\sss{\mathcal F}} \nu}{\nu-1} H(\vep)
	$$
and we can conclude that the events $A_t$ and $B_{t,\vep}$ are asymptotically independent, i.e.
	$$
	\lim_{t \to \infty} \P(A_t \cap B_{t,\vep}) = \lim_{t\to \infty} \P(A_t) \P(B_{t,\vep}).
	$$
\begin{proof}

We begin by proving \eqref{remaining-lifetime}. Indeed, write
	\begin{align*}
	\e^{-\lambda^{\sss{\mathcal F}}t} \E [Z^{\sss{\mathcal F}}_{t,\vep}] & = \sum_{k=0}^\infty \nu^k \e^{-\lambda^{\sss{\mathcal F}}t} 	
	\P(S_k^{\sss{\mathcal F}} \le t, S_{k+1}^{\sss{\mathcal F}} \in (t,t+\vep)) \\ 
	& = \sum_{k=0}^\infty \E\left[\e^{-\lambda^{\sss{\mathcal F}}(t-\bar{S}_k^{\sss{\mathcal F}})}; 
	\bar{S}_k^{\sss{\mathcal F}} \le t, \bar{S}_k^{\sss{\mathcal F}} + L_{k+1}^{\sss{\mathcal F}} \in (t,t+\vep) \right] \\
	& = \sum_{k=0}^\infty \int_{0}^t \e^{-\lambda^{\sss{\mathcal F}}(t-u)} 
	\P(L^{\sss{\mathcal F}} \in (t-u,t-u+\vep)) \P(\bar{S}_k^{\sss{\mathcal F}} \in du) \\
	& \to \frac{1}{\bar{\nu}} \int_0^\infty \e^{-\lambda^{\sss{\mathcal F}}z} \P(L^{\sss{\mathcal F}} \in (z,z+\vep)) dz,
	\end{align*}
due to the Key Renewal Theorem.

The limiting relation \eqref{non-blocked-individuals-remaining-lifetime} now follows from a slightly more involved version of the argument above. It is more convenient to work with the number $Z^{\sss{\mathcal F}}_{t,\vep} - N^{\sss{\mathcal F}}_{t,\vep}$ of those individuals reached by the fake-news branching process that have been overtaken by correct news at some point on the way. 
We now write
	\begin{align}
	\label{blocked-individuals}
	& \e^{-\lambda^{\sss{\mathcal F}}t} \E [Z^{\sss{\mathcal F}}_{t,\vep} - N^{\sss{\mathcal F}}_{t,\vep}] 
	= \sum_{k=0}^\infty \nu^k \e^{-\lambda^{\sss{\mathcal F}}t} 
	\P(S_k^{\sss{\mathcal F}} \le t, S_{k+1}^{\sss{\mathcal F}} \in (t,t+\vep), \sigma \le k) \\
	& = \sum_{k=0}^\infty \sum_{m=0}^k \nu^k \e^{-\lambda^{\sss{\mathcal F}}t} 
	\P(S_k^{\sss{\mathcal F}} \le t, S_{k+1}^{\sss{\mathcal F}} \in (t,t+\vep), \sigma = m)\nn\\
	& = \sum_{m=0}^\infty \nu^m \sum_{k=m}^\infty \nu^{k-m} \e^{-\lambda^{\sss{\mathcal F}}t} 
	\P(S_k^{\sss{\mathcal F}} \le t, S_{k+1}^{\sss{\mathcal F}} \in (t,t+\vep), \sigma = m)\nn\\
	& = \sum_{m=0}^\infty \nu^m \sum_{k=m}^\infty \nu^{k-m} \e^{-\lambda^{\sss{\mathcal F}}t} \int_0^t \P(S_k^{\sss{\mathcal F}} \le t, S_{k+1}^{\sss{\mathcal F}} \in (t,t+\vep), \sigma = m, S_m^{\sss{\mathcal F}}  \in du).\nn
	\end{align}
We next use that
	\eqan{
	&\P(S_k^{\sss{\mathcal F}} \le t, S_{k+1}^{\sss{\mathcal F}} \in (t,t+\vep), \sigma = m, S_m^{\sss{\mathcal F}}  \in du)\nn\\
	&\qquad=\P(\sigma=m, S_m^{\sss{\mathcal F}}  \in du)\P(S_{k-m}^{\sss{\mathcal F}} \le t-u, S_{k-m+1}^{\sss{\mathcal F}} \in (t-u,t-u+\vep)),\nn
	}
and let $l=m-k$, to rewrite the above as
	\begin{align*}
	&\sum_{m=0}^\infty \nu^m \int_0^t \P(\sigma=m, S_m^{\sss{\mathcal F}}  \in du) 
	\sum_{l=0}^\infty \nu^{l}\e^{-\lambda^{\sss{\mathcal F}}t} \P(S_{l}^{\sss{\mathcal F}} \le t-u, S_{l+1}^{\sss{\mathcal F}} \in (t-u,t-u+\vep))\\
	& = \int_{0}^t \sum_{m=0}^\infty \P(\bar{\sigma}=m, \bar{S}_m^{\sss{\mathcal F}}  \in du) 
	\sum_{l=0}^\infty \nu^{l}\e^{-\lambda^{\sss{\mathcal F}}(t-u)} 
	\P(S_{l}^{\sss{\mathcal F}} \le t-u, S_{l+1}^{\sss{\mathcal F}} \in (t-u,t-u+\vep))\\
	& = \int_{0}^t \sum_{m=0}^\infty \P(\bar{\sigma}=m, \bar{S}_m^{\sss{\mathcal F}}  \in du) 
	\e^{-\lambda^{\sss{\mathcal F}}(t-u)} \E[Z^{\sss{\mathcal F}}_{t-u,\vep}].
	\end{align*}
Note that, thanks to \eqref{remaining-lifetime}, for any $\delta > 0$, we can find $T$ such that, for all $s \ge T$,
	$$
	\left|\frac{\e^{-\lambda^{\sss{\mathcal F}}s} 
	\E[Z^{\sss{\mathcal F}}_{s,\vep}]}{\frac{1}{\bar{\nu}} \int_0^\infty \e^{-\lambda^{\sss{\mathcal F}}z} 
	\P(L^{\sss{\mathcal F}} \in (z,z+\vep)) dz}-1\right| \le \delta.
	$$
Hence, by Fatou's Lemma,
	\begin{align*}
	& \liminf_{t \to \infty} \int_{0}^t \sum_{m=0}^\infty \P(\bar{\sigma}=m, \bar{S}_m^{\sss{\mathcal F}}  \in du) 
	\e^{-\lambda^{\sss{\mathcal F}}(t-u)} \E[Z^{\sss{\mathcal F}}_{t-u,\vep}] \\ 
	& \ge \liminf_{t \to \infty} \int_{0}^{t-T} \sum_{m=0}^\infty \P(\bar{\sigma}=m, \bar{S}_m^{\sss{\mathcal F}}  \in du) 
	\e^{-\lambda^{\sss{\mathcal F}}(t-u)} \E[Z^{\sss{\mathcal F}}_{t-u,\vep}] \\
	& \ge (1-\delta) \frac{1}{\bar{\nu}} \int_0^\infty \e^{-\lambda^{\sss{\mathcal F}}z} 
	\P(L^{\sss{\mathcal F}} \in (z,z+\vep)) dz \liminf_{t \to \infty} \int_0^{t-T} 
	\sum_{m=0}^\infty \P(\bar{\sigma}=m, \bar{S}_m^{\sss{\mathcal F}}  \in du) \\ 
	& = (1-\delta) \frac{1}{\bar{\nu}} \int_0^\infty \e^{-\lambda^{\sss{\mathcal F}}z} \P(L^{\sss{\mathcal F}} \in (z,z+\vep)) dz 
	\liminf_{t \to \infty} \P(\bar{\sigma} < \infty, \bar{S}_{\bar{\sigma}}^{\sss{\mathcal F}} \le t-T) \\
	& \ge (1-\delta) (1-p^\star) \frac{1}{\bar{\nu}} \int_0^\infty \e^{-\lambda^{\sss{\mathcal F}}z} \P(L^{\sss{\mathcal F}} \in (z,z+\vep)) dz.
	\end{align*}
To provide the upper bound, we apply the inequality
	$$
	\e^{-\lambda^{\sss{\mathcal F}}(t-u)} \E[Z^{\sss{\mathcal F}}_{t-u,\vep}] 
	\le \e^{-\lambda^{\sss{\mathcal F}}(t-u)} \E[Z^{\sss{\mathcal F}}_{t-u}] \le C,
	$$
see \cite[Lemma 5.4 and (4.41)]{BhaHofHoo17}. With this,
	\begin{align*}
	& \limsup_{t \to \infty} \int_{0}^t \sum_{m=0}^\infty \P(\bar{\sigma}=m, \bar{S}_m^{\sss{\mathcal F}}  \in du) 
	\e^{-\lambda^{\sss{\mathcal F}}(t-u)} \E[Z^{\sss{\mathcal F}}_{t-u,\vep}] \\ 
	 = & \limsup_{t \to \infty}\biggl( \int_{0}^{t-T} \sum_{m=0}^\infty \P(\bar{\sigma}=m, \bar{S}_m^{\sss{\mathcal F}}  \in du) 
	 \e^{-\lambda^{\sss{\mathcal F}}(t-u)} \E[Z^{\sss{\mathcal F}}_{t-u,\vep}] \\ 
	&+ \int_{t-T}^{t} \sum_{m=0}^\infty \P(\bar{\sigma}=m, \bar{S}_m^{\sss{\mathcal F}}  \in du) \e^{-\lambda^{\sss{\mathcal F}}(t-u)} 
	\E[Z^{\sss{\mathcal F}}_{t-u,\vep}] \biggr)\\
 	\le & \limsup_{t \to \infty} \biggl((1+\delta) \frac{1}{\bar{\nu}} \int_0^\infty \e^{-\lambda^{\sss{\mathcal F}}z} 
	\P(L^{\sss{\mathcal F}} \in (z,z+\vep)) dz \int_0^{t-T} \sum_{m=0}^\infty \P(\bar{\sigma}=m, \bar{S}_m^{\sss{\mathcal F}}  \in du)\\
	& + C \int_{t-T}^t \sum_{m=0}^\infty \P(\bar{\sigma}=m, \bar{S}_m^{\sss{\mathcal F}}  \in du) \biggr)\\
	 = & \limsup_{t \to \infty} \biggl((1+\delta) \frac{1}{\bar{\nu}} \int_0^\infty \e^{-\lambda^{\sss{\mathcal F}}z} 
	 \P(L^{\sss{\mathcal F}} \in (z,z+\vep)) dz \int_0^{t-T} 
	 \sum_{m=0}^\infty \P(\bar{\sigma} < \infty, \bar{S}_{\bar{\sigma}}^{\sss{\mathcal F}} \le t-T) \\
	& + C \P(\bar{\sigma} < \infty, \bar{S}_{\bar{\sigma}}^{\sss{\mathcal F}} \in (t-T,t))\biggr)\\
	 = & (1+\delta) (1-p^\star) \frac{1}{\bar{\nu}} \int_0^\infty \e^{-\lambda^{\sss{\mathcal F}}z} \P(L^{\sss{\mathcal F}} \in (z,z+\vep)) dz.
	\end{align*}
Combining the above, we conclude
	\begin{equation} 
	\label{eq:limit_aux2}
	\e^{-\lambda^{\sss{\mathcal F}}t} \E [Z^{\sss{\mathcal F}}_{t,\vep} - N^{\sss{\mathcal F}}_{t,\vep}] 
	\to (1-p^\star) \frac{1}{\bar{\nu}} \int_0^\infty \e^{-\lambda^{\sss{\mathcal F}}z} \P(L^{\sss{\mathcal F}} \in (z,z+\vep)) dz,
	\end{equation}
as $t \to \infty$ and thus, taking \eqref{remaining-lifetime} into account, we finally obtain
	\begin{equation} 
	\label{eq:limit_ntvep}
	\lim_{t \to \infty} \e^{-\lambda^{\sss{\mathcal F}}t} \E [N^{\sss{\mathcal F}}_{t,\vep}] 
	= p^\star \frac{1}{\bar{\nu}} \int_0^\infty \e^{-\lambda^{\sss{\mathcal F}}z} \P(L^{\sss{\mathcal F}} \in (z,z+\vep)) dz.
	\end{equation}
\end{proof}

We next discuss an extension of Lemma \ref{lem:branching_process_survival}. For $\Time>0$ large, we let $N^{\sss{\mathcal F}}_{t,\Time}$ denote the number of individuals $v$ alive at time $t$ such that $S_k^{u, \sss{\mathcal F}} < S_k^{u, \sss{\mathcal R}}$ for every $u$ which is an ancestor of $v$, and $k$ such that $S_k^{u, \sss{\mathcal F}}\leq \Time$ (in other words, $N^{\sss{\mathcal F}}_{t,\Time}$ counts the number of individuals alive at time $t$ in the fake news epidemic who are not blocked by correct news before time $\Time$). Obviously, $N^{\sss{\mathcal F}}_{t}\leq N^{\sss{\mathcal F}}_{t,\Time}\leq Z^{\sss{\mathcal F}}_{t}$. The following corollary describes the asymptotics of $N^{\sss{\mathcal F}}_{t,\Time}:$

\begin{corollary}[Survival fake news in CTBP setting]
\label{cor:branching_process_survival}
Condition on the survival of the fake-news branching process, and assume that 
$\lambda^{\sss{\mathcal F}} > \lambda^{\sss{\mathcal R}}$, then
	\eqn{
	\label{non-blocked-individuals-remaining-lifetime-Time}
	\lim_{t \to \infty} \e^{-\lambda^{\sss{\mathcal F}}t} \E [N^{\sss{\mathcal F}}_{t,\Time}] = \frac{p^\star_{\sss\Time}}{\bar{\nu}},
	}
where
	\eqn{
	p^\star_{\sss\Time}\equiv p^\star+\sum_{m\geq 0} \P(\bar{\sigma}=m, \bar{S}_m^{\sss{\mathcal F}}>\Time).
	}
\end{corollary}

\proof The proof follows that of Lemma \ref{lem:branching_process_survival}, see e.g., \eqref{blocked-individuals}.  Note that now
$Z^{\sss{\mathcal F}}_{t} - N^{\sss{\mathcal F}}_{t,\Time}$ counts the number of individuals who are blocked before time $\Time$. Thus,
	\eqan{
	\e^{-\lambda^{\sss{\mathcal F}}t} \E [Z^{\sss{\mathcal F}}_{t} - N^{\sss{\mathcal F}}_{t,\Time}]
	&=\int_0^{\Time}\sum_{m=0}^\infty \P(\bar{\sigma}=m, \bar{S}_m^{\sss{\mathcal F}}  \in du) 
	\e^{-\lambda^{\sss{\mathcal F}}(t-u)} \E[Z^{\sss{\mathcal F}}_{t-u,\vep}].
	}
Taking the limit yields
	\eqan{
	\e^{-\lambda^{\sss{\mathcal F}}t} \E [Z^{\sss{\mathcal F}}_{t} - N^{\sss{\mathcal F}}_{t,\Time}]
	&\rightarrow \sum_{m\geq 0} \P(\bar{\sigma}=m, \bar{S}_m^{\sss{\mathcal F}}\leq \Time) 
	\frac{1}{\bar{\nu}}.
	}
Thus,
	\eqn{
	\e^{-\lambda^{\sss{\mathcal F}}t} \E [N^{\sss{\mathcal F}}_{t,\Time}]
	\rightarrow \frac{1}{\bar{\nu}}-\sum_{m\geq 0} \P(\bar{\sigma}=m, \bar{S}_m^{\sss{\mathcal F}}\leq \Time)
	=p^\star+\sum_{m\geq 0} \P(\bar{\sigma}=m, \bar{S}_m^{\sss{\mathcal F}}>\Time).
	}

\qed

We next adapt the above argument to the graph setting. Fix $\Time$ large, and split
	\eqn{
	\label{split-Nn-F}
	N_n^{\sss \mathcal{F}}=N_{n, {\sss <\Time}}^{\sss \mathcal{F}}-N_{n, {\sss \geq \Time}}^{\sss \mathcal{F}},
	}
where $N_{n, {\sss <\Time}}^{\sss \mathcal{F}}$ denotes the number of vertices reached by the potential fake news epidemic that are not killed by the correct news before time $\Time$, so that $N_{n, {\sss \geq \Time}}^{\sss \mathcal{F}}$ is the number of vertices reached by the potential fake news epidemic that are killed by the correct news after time $\Time$.

Lemma \ref{lem:branching_process_survival}, together with the powerful results from \cite{BhaHofHoo17}, can be used to prove that the first term in \eqref{split-Nn-F} remains positive:

\begin{lemma}[Probability of no early path blocking is positive]
\label{lem-early-path-blocking-positive-prob}
	\eqn{
	\liminf_{\Time\rightarrow \infty}\liminf_{n\rightarrow \infty} \frac{1}{n}\expec[N_{n, {\sss <\Time}}^{\sss \mathcal{F}}]=p^\star>0.
	}
\end{lemma}

\proof Rewrite
	\eqn{
	\frac{1}{n}\expec[N_{n, {\sss <\Time}}^{\sss \mathcal{F}}]
	=\prob(\text{fake news from }\Ver_1 \text{ to }\Ver_2\text{ is not path-blocked before time }\Time), 
	}
where $\Ver_1$ and $\Ver_2$ are two vertices chosen uniformly at random from $[n]$.
\medskip

\paragraph{\bf Overview: good and bad collision edges.}
We rely on the analysis in \cite{BhaHofHoo17}, in particular the proof of \cite[Theorem 3.1]{BhaHofHoo17}. This theorem gives a precise analysis of the number of {\em collision edges} between the smallest-weight graphs (SWG) of fake news from $\Ver_1$ and $\Ver_2$. Here, \cite[Theorem 3.1]{BhaHofHoo17} conditions on these smallest-weight trees started from $\Ver_1$ and $\Ver_2$, respectively, up to time $s_n$ for some $s_n\rightarrow \infty$. These collision edges, and the times at which they occur, form an inhomogeneous Poisson process. Further, whp, the two smallest-weight trees started from $\Ver_1$ and $\Ver_2$ are {\em perfectly coupled} to the limiting branching processes up to time $s_n$ for some $s_n\rightarrow \infty$. \cite[Theorem 3.1]{BhaHofHoo17} states that the times at which these collision edge are formed, as well as the residual life-time on the collision edge (together with various other quantities) forms an inhomogeneous Poisson process. 


Consider the SWG of vertex $\Ver_1$ at time $s_n$, as well as the alive half-edges incident to it. We give each of these alive half-edges a {\em type}, which is {\em good} when the half-edge has not been path-blocked before time $\Time$, and {\em bad} when it has been path-blocked before time $\Time$. Below, we follow the analysis of the growth of the SWGs from $\Ver_1$ and $\Ver_2$ as in \cite[Section 6]{BhaHofHoo17}, as well as how the collision edges appear as a Poisson process in \cite[Section 7]{BhaHofHoo17}.

We give each collision edge a type that is the same as the type of the half-edge incident to the SWG of vertex $\Ver_1$. The smallest-weight path is obtained by taking the {\em optimal collision edge}, i.e. that collision edge for which the path between $\Ver_1$ and $\Ver_2$ that it completes, has minimal total weight. In terms of these notions,
	\eqn{
	\frac{1}{n}\expec[N_{n, {\sss <\Time}}^{\sss \mathcal{F}}]
	=\prob(\text{optimal collision edge is good}).
	}
We claim that this probability converges to $p_{\sss \Time}^\star>0$. We prove this below.
\medskip

\paragraph{\bf Close to deterministic growth of SWG after time $s_n$ and a useful martingale.}
In \cite[Section 6]{BhaHofHoo17}, it is explained that the active half-edges at time $t$ basically grow deterministically after time $s_n$. This is a crucial ingredient in the Poisson process limit proof performed in \cite[Section 7]{BhaHofHoo17}, where we apply it for times around some appropriately chosen $\bar{t}_n$. We now highlight this proof in order to investigate the effect of the types of the alive half-edges. \cite[Section 6]{BhaHofHoo17} proves \cite[Proposition 2.3]{BhaHofHoo17}, which describes the evolution of the SWG. It is proved that this dynamics is close to {\em deterministic} after time $s_n$, which is proved by the following two key ingredients:
\begin{itemize}
\item[(a)] a precise coupling result between the SWG dynamics and that of an $n$-dependent continuous-time branching process that follows from \cite[Proposition 2.2]{BhaHofHoo17}, and can be invoked now too; and
\item[(b)] a  conditional second moment method on the branching process arising in step (a).
\end{itemize}
Let $t_n=\log{n}/[2\lambda_n^{\sss{\mathcal F}}]$, and define $\bar{t}_n=t_n-\log{\Big(\mathcal{W}^{\sss{(1),\mathcal F}}\mathcal{W}^{\sss{(2),\mathcal F}}\Big)}/[2\lambda_n^{\sss{\mathcal F}}]$, where $(\mathcal{W}^{\sss{(1),\mathcal F}}_{s_n},\mathcal{W}^{\sss{(2),\mathcal F}}_{s_n})$ are the total number of half-edges incident to the potential fake-news SWG at times $s_n$. Let $\SWG^{\sss(j)}[\bar{t}_n+t, \bar{t}_n+t+s)$ denote the number of active half-edges in SWG of vertex $U_j$ at time $\bar{t}_n+t$ and residual life time in $[0,s]$.
In particular, the above two steps show that $\SWG^{\sss(j)}[\bar{t}_n+t, \bar{t}_n+t+s)$ satisfies (see \cite[(6.2)]{BhaHofHoo17})
	\eqn{
	\e^{-\lambda_n^{\sss{\mathcal F}} t_n}|\SWG^{\sss(j)}[\bar{t}_n+t, \bar{t}_n+t+s)|\convp \e^{\lambda^{\sss{\mathcal F}} t} F_{\sss R}(s)\sqrt{\Wcal^{\sss(j)}/\Wcal^{\sss(3-j)}}.
	}
The way this is proved is by first showing that $|\SWG^{\sss(j)}[\bar{t}_n+t, \bar{t}_n+t+s)|$ is close to $|\BP^{\sss(j)}_n[\bar{t}_n+t, \bar{t}_n+t+s)|$ as indicated in (a) (see \cite[first equality in (6.2)]{BhaHofHoo17}), and then proving a conditional second moment as in (b) to reach the conclusion in \cite[(6.2)]{BhaHofHoo17}.

We wish to adapt this argument so as to include the {\em types} of half-edges in step (b), i.e., whether the half-edges are {\em good} or {\em bad}. In step (b), what is being proved is that (see \cite[(6.14)]{BhaHofHoo17}), for some constant $A^{\sss{\mathcal F}}$,
	\eqn{
	\label{condit-moment-SWG-sn}
	\e^{-\lambda_n^{\sss{\mathcal F}} t_n}|\BP^{\sss(j)}_n[\bar{t}_n+t, \bar{t}_n+t+s)|
	=(1+\op(1))A^{\sss{\mathcal F}}\e^{\lambda^{\sss{\mathcal F}} t} F_{\sss R}(s) \e^{\lambda_n^{\sss{\mathcal F}}(\bar{t}_n-t_n-s_n)} \sum_{i\in \BP(s_n)}\e^{-\lambda_n^{\sss{\mathcal F}} R_i}.
	}
This is shown by a conditional second moment argument. Indeed, the conditional expectation of $\BP^{\sss(j)}[\bar{t}_n+t, \bar{t}_n+t+s)$ given $\mathscr{F}_{s_n}$ is computed to be equal to the rhs of \eqref{condit-moment-SWG-sn}. Then, by the branching property of $\BP^{\sss(j)}_n[\bar{t}_n+t, \bar{t}_n+t+s)$, it is shown that the conditional variance is $\op(1)$ times the rhs of \eqref{condit-moment-SWG-sn}, which shows that \eqref{condit-moment-SWG-sn} holds. Note that, since 
	\eqn{
	\label{martingale-CTBP-1}
	M(s)=\e^{-\lambda^{\sss{\mathcal F}} s}\sum_{i\in \BP(s)}\e^{-\lambda^{\sss{\mathcal F}} R_i}
	}
arises as
	\eqn{
	\label{martingale-CTBP-2}
	M(s)=\lim_{t\rightarrow \infty} \e^{-\lambda^{\sss{\mathcal F}} t}\expec\big[|\BP(t)|\mid \mathscr{F}_{s}\big],
	}
it immediately follows that $(M(s))_{s\geq 0}$ is a non-negative martingale. That, and some extensions of this observation, will be useful below.
\medskip

\paragraph{\bf Incorporating good and bad types of half-edges.}
We next discuss what happens when we keep track of the {\em types} of half-edges. Let $|\BP^{\sss(1)}_{n, \msg}[\bar{t}_n+t, \bar{t}_n+t+s)|$ denote the number of alive {\em good} half-edges at time $\bar{t}_n+t$ with residual life time in $[0,s]$. Applying the almost deterministioc growth argument explained below \eqref{condit-moment-SWG-sn}, we obtain that
	\eqn{
	\label{condit-moment-SWG1-sn-good}
	\e^{-\lambda_n^{\sss{\mathcal F}} t_n}|\BP^{\sss(j)}_{n, \msg}[\bar{t}_n+t, \bar{t}_n+t+s)|
	=(1+\op(1))A^{\sss{\mathcal F}}\e^{\lambda^{\sss{\mathcal F}} t} F_{\sss R}(s) \e^{\lambda_n^{\sss{\mathcal F}}(\bar{t}_n-t_n-s_n)} \sum_{i\in \BP_{\msg}(s_n)}\e^{-\lambda_n^{\sss{\mathcal F}} R_i},
	}
where, for $s>\Time$, $\BP_{\msg}(s)$ denotes the collection of alive good half-edges at time $s$. By the argument in \eqref{martingale-CTBP-1}--\eqref{martingale-CTBP-2}, it follows that, with
	\eqn{
	M_{\msg}(s)=\e^{-\lambda^{\sss{\mathcal F}} s}\sum_{i\in \BP_{\msg}(s)}\e^{-\lambda^{\sss{\mathcal F}} R_i},
	}
the process $(M_{\msg}(s))_{s\geq \Time}$ is also a non-negative martingale. This is only true for $s\geq \Time$, since the good half-edges after time $\Time$ only arise as descendants of good half-edges at time $\Time$. 

By the above argument, we obtain that the proportion of half-edges at time $\bar{t}_n+t$ with residual life-time in $[0,s]$ that is good equals
	\eqn{
	\frac{M_{\msg}(s_n)}{M(s_n)}(1+\op(1)).
	}
Further, both terms in the ratio are non-negative martingales, and thus converge a.s. Let $\Mcal_{\msg}$ and $\Mcal$ denote their limits, so that the asymptotic proportion of good half-edges is close to
	\eqn{
	P^\star_{\sss \Time}= \Mcal_{\msg}/\Mcal.
	}
Then, we conclude that the growth of the number of good and bad half-edges in $|\SWG^{\sss(1)}[\bar{t}_n+t, \bar{t}_n+t+s)|$ is close to deterministic, with the proportion of good half-edges being close to $P^\star_{\sss \Time}$, and the residual life-time distribution of the good half-edges being equal to that of all half-edges. By showing that $\E[P^\star_{\sss \Time}]=p^\star_{\sss \Time}$, we can then conclude the argument. Let us now provide the details.
\medskip

\paragraph{\bf A Poisson-process limit of good and total collision edges.}
Since the proof of \cite[Theorem 3.1]{BhaHofHoo17} in \cite[Section 7]{BhaHofHoo17} is quite relevant here, we will give a brief update on how this proof is organized. \cite[Theorem 3.1]{BhaHofHoo17} involves the properties of the collision edges between the smallest-weight trees of fake news from $\Ver_1$ and $\Ver_2$. Recall that time is re-centered as $\bar{t}_n=t_n-\log{\Big(\mathcal{W}^{\sss{(1),\mathcal F}}\mathcal{W}^{\sss{(2),\mathcal F}}\Big)}/[2\lambda_n^{\sss{\mathcal F}}]$, where $(\mathcal{W}^{\sss{(1),\mathcal F}}_{s_n},\mathcal{W}^{\sss{(2),\mathcal F}}_{s_n})$ are the total number of half-edges incident to the potential fake-news SWG at times $s_n$ and $t_n=\log{n}/[2\lambda_n^{\sss{\mathcal F}}]$. Further, whp, the SWGs at this time are {\em perfectly coupled} to two independent CTBPs. The reason for this alternative centering is that the limiting Poisson process of potential collision edges in \cite[Theorem 3.1]{BhaHofHoo17} becomes {\em deterministic}. Indeed, due to the fact that we are interested in collision edges {\em between} the two SWGs, the rate at time $t$ for such edges involves a factor $\mathcal{W}^{\sss{(1),\mathcal F}}\mathcal{W}^{\sss{(2),\mathcal F}}$. By the re-centering, such factors drop out, which simplifies the analysis considerably.

The proof in \cite[Section 7]{BhaHofHoo17} relies on the fact that the evolution of the SWGs is basically {\em deterministic} as argued above, and edges between the two start appearing when they both have size $\Theta_{\sss \prob}(\sqrt{n})$. \cite[Theorem 3.1]{BhaHofHoo17} contains several more features that are irrelevant to us (such as the number of edges in each of the two SWGs along the path leading up to the collision edge), or may be irrelevant (such as whether the edge is formed by a depleted half-edge from the SWG of $\Ver_1$ pairing to a half-edge incident to the SWG of $\Ver_2$ or vice versa). However, we instead wish to keep track of the number of half-edges incident to the SWG of $\Ver_1$ that are not path blocked before time $\Time$.

Due to the deterministic nature of the growth {\em both} of the number of good and bad half-edges, the process of times where good collision edges appear, and the process of times where bad collision edges appear, are {\em independent} Poisson processes with intensity $P^\star_{\sss \Time}\Pi$ and $(1-P^\star_{\sss \Time})\Pi$, respectively, where $\Pi$ is the total intensity of collision edges as identified in \cite[Theorem 3.1]{BhaHofHoo17}. 
\medskip

\paragraph{\bf Conclusion of the proof.}
The optimal collision edge solves an optimization problem in which its birth time as well as the residual life-time of the half-edge appears. Since both are {\em independent} of the good or bad type, we conclude that the optimal collision edge will be good with probability $P^\star_{\sss \Time}(1+\op(1))$, so that
	\eqan{
	\lim_{n\rightarrow \infty} \frac{1}{n}\expec[N_{n, {\sss <\Time}}^{\sss \mathcal{F}}]
	&=\lim_{n\rightarrow \infty}\prob(\text{optimal collision edge is good})\\
	&=\lim_{n\rightarrow \infty}\expec\Big[\prob(\text{optimal collision edge is good}\mid \mathscr{F}_{s})\Big]\nn\\
	&=\lim_{n\rightarrow \infty}\expec\big[P^\star_{\sss \Time}(1+\op(1))\big]=\expec[P^\star_{\sss \Time}],\nn
	}
by dominated convergence. The result now follows from Corollary \ref{cor:branching_process_survival}, which implies that
	\eqn{
	\label{large-A-limit-BP}
	\expec[P^\star_{\sss \Time}]=p^\star_{\sss \Time}.
	}
Indeed, for the fake and correct news epidemics on the branching process tree, we can perform the exact same analysis as above, to obtain that now (recall the definition of $N^{\sss{\mathcal F}}_{t,\Time}$ above Corollary \ref{cor:branching_process_survival})
	\eqn{
	\nn
	\e^{-\lambda^{\sss{\mathcal F}}t}N^{\sss{\mathcal F}}_{t,\Time}\convas \Mcal_{\msg}/\Mcal
	}
Since this convergence is also in $L^1$, 
	\eqn{
	\nn
	\e^{-\lambda^{\sss{\mathcal F}}t}\expec[N^{\sss{\mathcal F}}_{t,\Time}]\to \expec[\Mcal_{\msg}/\Mcal].
	}
Comparing this to Corollary \ref{cor:branching_process_survival}, we have that $\expec[\Mcal_{\msg}/\Mcal]=p^\star_{\sss \Time}$. Finally, using that 
	$$
	\lim_{\Time\rightarrow \infty}p^\star_{\sss \Time}=p^\star
	$$
completes the proof.
\qed
\medskip

The following lemma shows that a late path-blocking is unlikely:

\begin{lemma}[Probability of late path blocking vanishes]
\label{lem-late-path-blocking-unlikely}
	\eqn{
	\liminf_{\Time\rightarrow \infty}\limsup_{n\rightarrow \infty} \frac{1}{n}\expec[N_{n, {\sss \geq \Time}}^{\sss \mathcal{F}}]=0.
	}
\end{lemma}


\proof 
For this, we will crucially rely on the analysis in \cite[Section 4]{BhaHofHoo17}, where, under the assumptions of Theorem \ref{thm-strong-CM}, the expected number of alive individuals in the CTBP are computed rather precisely relying on the key renewal theorem. This proof is inspired by Samuels \cite{Samu71}, and is presented rather nicely by Harris in \cite[Chapter VI]{Harr63}. It is this argument that is crucially used in \cite[Section 4]{BhaHofHoo17}, where it is adapted to apply under the $X\log{X}$ condition that follows from \eqref{cond-D2logD}.

We split
	\eqn{
	 \frac{1}{n}\expec[N_{n, {\sss \geq \Time}}^{\sss \mathcal{F}}]= \frac{1}{n}\expec[N_{n, {\sss \geq \Time}}^{\sss \mathcal{F}}(1)]+ \frac{1}{n}\expec[N_{n, {\sss \geq \Time}}^{\sss \mathcal{F}}(2)],
	}
where the split depends on whether the potential fake news epidemics reaches $\Ver'$ before time $t_n^+$, where $t_n^+=\log{n}/\lambda_n^{\sss{\mathcal F}}+B$, for some $B$ large. The contribution $\expec[N_{n, {\sss \geq \Time}}^{\sss \mathcal{F}}(2)]/n$ where this occurs later than $t_n^+$ vanishes when $B\rightarrow \infty$, so that
	\eqn{
	 \liminf_{B\rightarrow \infty}\liminf_{n\rightarrow \infty}\frac{1}{n}\expec[N_{n, {\sss \geq \Time}}^{\sss \mathcal{F}}(2)]=0,
	 }
uniformly in $\Time$. Thus, we can fix $B=B_\vep$ such that $\frac{1}{n}\expec[N_{n, {\sss \geq \Time}}^{\sss \mathcal{F}}(2)]\leq \vep/2$ for all $n$ large.

We continue with $\expec[N_{n, {\sss \geq \Time}}^{\sss \mathcal{F}}(1)]$. By \cite[Lemma 5.1]{Jans09b}, the number of $k$-step paths $P_k$ satisfies
	\eqn{
	\expec[P_k] \leq n\expec[D_n] \nu_n^{k-1}.
	}
Thus, 
	\eqan{
	\label{Nn-large-mean}
	\frac{1}{n}\E[N_{n, {\sss \geq \Time}}^{\sss \mathcal{F}}(1)]
	&\leq \frac{\expec[D_n]}{n\nu_n}\sum_{k=0}^\infty \nu^k_n \P(S_{k,n}^{\sss{\mathcal F}} \leq t_n^+, \Time\leq S_{m,n}^{\sss{\mathcal R}} < S_{m,n}^{\sss{\mathcal F}} \text{ for some }m\in [k]).
	}
We define the stopping time
	\eqn{
	\sigma=\inf\big\{m\colon \Time \leq S_{m,n}^{\sss{\mathcal R}} < S_{m,n}^{\sss{\mathcal F}}\big\}
	}
to denote the first $m$ for which $\Time\leq S_{m,n}^{\sss{\mathcal R}} < S_{m,n}^{\sss{\mathcal F}}$. We condition on $\sigma=m$ and on $S_{m,n}^{\sss{\mathcal F}}$, and split
	\eqan{
	\label{Nn-large-mean-b}
	\frac{1}{n}\E[N_{n, {\sss \geq \Time}}^{\sss \mathcal{F}}(1)]
	&\leq \frac{\expec[D_n]}{n\nu_n}\sum_{m,k=0}^\infty \nu^k_n \expec\Big[\P(S_{k,n}^{\sss{\mathcal F}} \leq t_n^+\mid \sigma=m, S_{m,n}^{\sss{\mathcal F}})\Big]\\
	&=\sum_{m=0}^\infty \nu^m_n \expec\Big[\sum_{k\geq m}\nu_n^{k-m} \indic{\sigma=m} \P(S_{k,n}^{\sss{\mathcal F}} \leq t_n^+\mid \sigma=m, S_{m,n}^{\sss{\mathcal F}})\Big].\nn
	}
By the Markov property,
	\eqan{
	\sum_{k\geq m}\nu_n^{k-m} \P(S_{k,n}^{\sss{\mathcal F}} \leq t_n^+\mid \sigma=m, S_{m,n}^{\sss{\mathcal F}})
	&=\sum_{k\geq 0}\nu_n^k \P(\tilde S_{k,n}^{\sss{\mathcal F}} \leq t_n^+-S_{m,n}^{\sss{\mathcal F}}\mid S_{m,n}^{\sss{\mathcal F}}),
	}
where $(\tilde S_{k,n}^{\sss{\mathcal F}})_{k\geq 0}$ is an independent random walk. We next use that
	\eqn{
	\sum_{k\geq 0}\nu_n^k \P(\tilde S_{k,n}^{\sss{\mathcal F}} \leq t)=\E\left[Z^*_n(t)\right],
	}
where $Z^*_n(t)$ denotes the total number of individuals that have ever been alive up to time $t$ in the CTBP. See \cite[Lemma 5.4 and (4.41)]{BhaHofHoo17} for an analysis of this quantity that shows that 
	\eqn{
	\E\left[Z^*_n(t)\right]\leq C\e^{\lambda_n^{\sss{\mathcal F}} t}.
	}
We apply this to $t=t_n^+-S_{m,n}^{\sss{\mathcal F}}$, which yields
	\eqn{
	 \E\left[Z^*_n(t_n^+-S_{m,n}^{\sss{\mathcal F}})|S_{m,n}^{\sss{\mathcal F}}\right]\leq C n \e^{\lambda_n^{\sss{\mathcal F}}(B-S_{m,n}^{\sss{\mathcal F}})} \quad \text{a.s.}
	 }
This leads us to 
	\eqan{
	\label{Nn-large-mean-c}
	\frac{1}{n}\E[N_{n, {\sss \geq \Time}}^{\sss \mathcal{F}}(1)]
	&\leq C\e^{\lambda_n^{\sss{\mathcal F}}B}\sum_{m=0}^\infty \nu^m_n \expec\Big[\e^{-\lambda_n^{\sss{\mathcal F}}S_{m,n}^{\sss{\mathcal F}}}
	\indic{\sigma=m}\Big]\\
	&=C\e^{\lambda_n^{\sss{\mathcal F}}B}\sum_{m=0}^\infty \P(\bar{\sigma}=m),\nn
	}
where now 
	\eqn{
	\bar{\sigma}=\inf\big\{m\colon \Time\leq \bar{S}_{m,n}^{\sss{\mathcal R}} < \bar{S}_{m,n}^{\sss{\mathcal F}}\big\}.
	}
Recall the definitions of $\bar{S}^{\sss{\mathcal R}}$ and $\bar{S}^{\sss{\mathcal F}}$ before Lemma \ref{lem:branching_process_survival}.  We end up with
	\eqan{
	\frac{1}{n}\E[N_{n, {\sss \geq \Time}}^{\sss \mathcal{F}}(1)]
	&\leq C\e^{\lambda_n^{\sss{\mathcal F}}B}\P(\bar{\sigma}<\infty)
	=C\e^{\lambda_n^{\sss{\mathcal F}}B}\P(\exists m\colon \Time\leq \bar{S}_{m,n}^{\sss{\mathcal R}} < \bar{S}_{m,n}^{\sss{\mathcal F}}).
	}
When $n\rightarrow \infty$, 
	\eqan{
	\limsup_{n\rightarrow \infty} \frac{1}{n}\E[N_{n, {\sss \geq \Time}}^{\sss \mathcal{F}}(1)]
	&\leq C\e^{\lambda^{\sss{\mathcal F}}B}\P(\exists m\colon \Time\leq \bar{S}_{m}^{\sss{\mathcal R}} < \bar{S}_{m}^{\sss{\mathcal F}}).
	}
Since $\expec[\bar{S}_{m}^{\sss{\mathcal R}}]>\expec[ \bar{S}_{m}^{\sss{\mathcal F}}]$ by Lemma \ref{lem-sad-correct}, the above vanishes as $\Time\rightarrow \infty$. Thus, for a given $B=B_{\vep}$, we can take $\Time=\Time_\vep(B)$ so large that $\P(\exists m\colon \Time \leq \bar{S}_{m}^{\sss{\mathcal R}} < \bar{S}_{m}^{\sss{\mathcal F}})\leq \vep/2$. Therefore, we can take $\Time=\Time_\vep$ so large that
	\eqn{
	\limsup_{n\rightarrow \infty} \frac{1}{n}\expec[N_{n, {\sss \geq \Time}}^{\sss \mathcal{F}}]\leq \vep/2+\vep/2=\vep.
	}
This completes the proof of Lemma \ref{lem-late-path-blocking-unlikely}.
%
%
\qed
\medskip

\paragraph{\bf Completion of the proof of Proposition \ref{prop-exposed-fake-news}.} The above arguments show that \eqref{aim-prop-exposed-fake-news-2} indeed holds, which completes the proof of Proposition \ref{prop-exposed-fake-news}.
\end{proof}



\subsection{Proof of no strong survival in Theorem \ref{thm-strong-CM} when $\lambda^{\sss \mathcal{R}}>\lambda^{\sss \mathcal{F}}$}
\label{sec-no-strong-CM}

We prove the second statement of Theorem \ref{thm-strong-CM}. A vertex is definitely not infected by fake news if the shortest path of the correct news to the vertex is shorter than the shortest path of the fake news. Hence,
	$$
	\frac{1}{n}\E\left[N_n^{\sss \mathcal{F}}\right] \le p_n,
	$$
where $p_n = \P(M_n^{\sss \mathcal{F}} < M_n^{\sss \mathcal{R}})$ and $M_n^{\sss \mathcal{F}/\mathcal{R}}$ are the weights of the shortest-weight paths for the fake and correct news, respectively, between $2$ typical points.

Thanks to \cite[Theorem 1.2]{BhaHofHoo17},
	$$
	M_n^{\sss \mathcal{F}} - \frac{\log n}{\lambda_n^{\sss \mathcal{F}}} \overset{d}{\to} Q^{\sss \mathcal{F}},
	$$
with a proper continuous random variable $Q^{\sss \mathcal{F}}$ and with a sequence $\lambda_n^{\sss \mathcal{F}} \to \lambda^{\sss \mathcal{F}}$ as $n \to \infty$. A similar result holds for correct news. As $\lambda^{\sss \mathcal{F}} < \lambda^{\sss \mathcal{R}}$, we can choose large values of $n$ such that $\lambda_n^{\sss \mathcal{F}} < \lambda_n^{\sss \mathcal{R}}$.
Then
	\begin{align*}
	p_n  = & \P\left(\big(M_n^{\sss \mathcal{F}} - \frac{\log n}{\lambda_n^{\sss \mathcal{F}}}\big) 
	- \big(M_n^{\sss \mathcal{R}} - \frac{\log n}{\lambda_n^{\sss \mathcal{R}}}\big) < \log n \big(\frac{1}{\lambda_n^{\sss \mathcal{R}}} - \frac{1}{\lambda_n^{\sss \mathcal{F}}}\big)\right)\\
	 \le &  \P\left(\big(M_n^{\sss \mathcal{F}} - \frac{\log n}{\lambda_n^{\sss \mathcal{F}}}\big) < \tfrac{1}{2} \log n \big(\frac{1}{\lambda_n^{\sss \mathcal{R}}} - \frac{1}{\lambda_n^{\sss \mathcal{F}}}\big)\right)\\
	 &\quad+ \P\left(-\big(M_n^{\sss \mathcal{R}} - \frac{\log n}{\lambda_n^{\sss \mathcal{R}}}\big) < \tfrac{1}{2} \log n \big(\frac{1}{\lambda_n^{\sss \mathcal{R}}} - \frac{1}{\lambda_n^{\sss \mathcal{F}}}\big)\right),
	\end{align*}
and the above tends to $0$, as the left-hand sides of the expressions inside the probabilities tend to proper random variables, while the right-hand sides tend to $-\infty$. The first-moment method then finishes the proof.

We next extend the above argument, and show that in this setting, the fake news is at most $n^{\lambda^{\sss \mathcal{F}}/\lambda^{\sss \mathcal{R}}+\vep}$-intermediate surviving for any $\vep>0$:

\begin{theorem}[Intermediate survival]
\label{thm-intermediate}
Fix $\vep>0$ and assume that $\lambda^{\sss \mathcal{R}}>\lambda^{\sss \mathcal{F}}$. Let the assumptions from Theorem \ref{thm-strong-CM} hold. Then, whp fake news is at most $\Theta_{\sss \prob}(1)n^{\lambda_n^{\sss \mathcal{F}}/\lambda_n^{\sss \mathcal{R}}}$-intermediate surviving.
\end{theorem}

\proof By \cite{BhaHofKom14}, and under the assumptions of Theorem \ref{thm-strong-CM},
	\eqn{
	\label{epid-curve} 
	\frac{N_n^{\sss \mathcal{R}} \Big(t + \frac{\log{(n/\mathcal{W}^{\sss \mathcal{R}}_{s_n}})}{\lambda_n^{\sss \mathcal{R}}}\Big)}{n} 
	\convp G(t),
	}
for some distribution function $G$ with support $\mathbb{R}$. 
\smallskip

Note that $\lambda_n^{\sss \mathcal{R}}\rightarrow \lambda^{\sss \mathcal{R}}$ and $\mathcal{W}^{\sss \mathcal{R}}_{s_n}\convd \mathcal{W}^{\sss \mathcal{R}}$. Therefore, at time $T_n = \frac{1+\varepsilon}{\lambda^{\sss \mathcal{R}}} \log n$,
	\eqn{
	\label{Nn-conv}
	\frac{N_n^{\sss \mathcal{R}}(T_n)}{n} \convp 1.
	}

Let $\tilde{N}_n^{\sss \mathcal{F}}(t)$ denote the potential number of vertices reached by fake news at time $t$, i.e., in a model without competition, as defined in Section \ref{sec-strong-CM}. Then, a.s.,
	\eqn{
	N_n^{\sss \mathcal{F}}(\frac{\log n}{\lambda_n^{\sss \mathcal{R}}}) \le \tilde{N}_n^{\sss \mathcal{F}} (\frac{\log n}{\lambda_n^{\sss \mathcal{R}}}).
	}
In fact, by  \cite{BhaHofHoo17} and as in \eqref{Nn-large-mean}, now with $\tilde{t}_n=\frac{\log n}{\lambda_n^{\sss \mathcal{R}}}$,
	\eqan{
	\expec\big[\tilde{N}_n^{\sss \mathcal{F}} (\frac{\log n}{\lambda_n^{\sss \mathcal{R}}})\big]
	&\leq \frac{\expec[D_n]}{n\nu_n}\sum_{k=0}^\infty \nu^k_n \P(S_k^{\sss{\mathcal F}} \leq \tilde{t}_n)\\
	&=\frac{\expec[D_n]}{n\nu_n}\sum_{k=0}^\infty \expec\Big[\e^{\lambda_n^{\sss{\mathcal F}}\bar{S}_k^{\sss{\mathcal F}}}
	\indic{\bar{S}_k^{\sss{\mathcal F}}\leq \tilde{t}_n}\Big]\nn\\
	&=n^{\lambda_n^{\sss \mathcal{F}}/\lambda_n^{\sss \mathcal{R}}} \frac{\expec[D_n]}{n\nu_n}\sum_{k=0}^\infty \expec\Big[\e^{\lambda_n^{\sss{\mathcal F}}(\bar{S}_k^{\sss{\mathcal F}}-\tilde{t}_n)}
	\indic{\bar{S}_k^{\sss{\mathcal F}}\leq \tilde{t}_n}\Big]\nn\\
	&=O(1)n^{\lambda_n^{\sss \mathcal{F}}/\lambda_n^{\sss \mathcal{R}}},\nn
	}
by \cite[Lemma 5.4 and (5.38)]{BhaHofHoo17}. Therefore, $n^{-\lambda_n^{\sss \mathcal{F}}/\lambda_n^{\sss \mathcal{R}}}\tilde{N}_n^{\sss \mathcal{F}} (\frac{\log n}{\lambda_n^{\sss \mathcal{R}}})$ is a tight sequence of random variables. As a result, the fake news is at most $\Theta_{\sss \prob}(1)n^{\lambda_n^{\sss \mathcal{F}}/\lambda_n^{\sss \mathcal{R}}}$-intermediate surviving, and this completes the proof.
%
%
\qed

\subsection{Strong survival for explosive fake news: Proof of Theorem \ref{thm-strong-CM-explos}}
\label{sec-strong-CM-explos}
In the explosive setting, we know that the time for the fake news (in the absence of correct news) to reach the hubs converges in distribution to a random variable $V^{\sss {\mathcal{F}}}$, which is finite a.s. After the fake news reaches the hubs, it will almost instantaneously reach a positive proportion of the vertices. This is described by the so-called {\em epidemic curve}. Since the time for the correct news to reach the hubs is larger than $V^{\sss {\mathcal{F}}}$ plus any constant with strictly positive probability, the correct news cannot easily block the fake news. We conclude that strong survival holds. We now explain this proof in more detail.

We will show that, at time $V^{\sss {\mathcal{F}}}+t$, a positive proportion of the vertices will be found with positive probability by \cite{BarHofKom17}, and the existence of the so-called {\em epidemic curve}.  The epidemic curve describes how the fake news invades the graph right after time $T_{a_n}^{\sss {\mathcal{F}}}$, which is the time for the fake news to reach $a_n$ vertices in the graph for the first time (with $a_n$ chosen appropriately). Let us introduce some notation. We will see that at time $T_{a_n}^{\sss {\mathcal{F}}}+t$, a positive proportion of the vertices is found by the first-passage percolation flow. To describe how the winning type sweeps through the graph, we need some notation. Write $\barNt$ for the fraction of vertices that have been reached by the flow at time $T_{a_n}^{\sss {\mathcal{F}}}+t$, that is,
	\eqn{
	\barNt=\#\{v\colon v\in \SWG(T_{a_n}^{\sss {\mathcal{F}}}+t)\}/n,
	}
where $\SWG(t)$ denotes the set of alive half-edges in the fake news epidemic.
The epidemic curve follows from the following proposition:

\begin{proposition}[Epidemic curve]
\label{prop-degree-time-FPP}
As $n\rightarrow \infty$,
	\eqn{
	\label{Ntk-conv-FPP}
	\barNt\convp \prob(V^{\sss {\mathcal{F}}}\leq t).
	}
\end{proposition} 

Proposition \ref{prop-degree-time-FPP} is interesting in its own right, as it implies the existence of an {\em epidemic curve}, as also derived in \cite{BhaHofKom14} in the setting of first-passage percolation on the CM with finite-variance degrees. Indeed, \eqref{Ntk-conv-FPP} together with the convergence $T_{a_n}^{\sss {\mathcal{F}}}\convd V^{\sss {\mathcal{F}}}$ shows that the fraction of invaded vertices at time $t$ (or infected vertices when interpreting first-passage percolation as a model for disease spread) converges in distribution to $\prob(V^{\sss {\mathcal{F}}}_2<t-V\mid V^{\sss {\mathcal{F}}})=F_{\sss V^{\sss {\mathcal{F}}}}(t-V^{\sss {\mathcal{F}}})$ for every $t\in {\mathbb R}$. Since $F_{\sss V^{\sss {\mathcal{F}}}}(s)>0$ for every $s>0$, the above result implies that at time $t+V^{\sss {\mathcal{F}}}$, the fake news indeed reaches a positive proportion of the vertices in the graph. This will be an essential ingredient to also show that there is strong survival. We refer to \cite{BhaHofKom14} for an extensive discussion of the consequences of an epidemic-curve result.

\proof This proof is a minor modification of \cite[Proof of Proposition 4.2]{DeiHof16}. Let $\Ver_1$ be a randomly chosen vertex and write $\indicwo{\sss \Ver_1}(t)$ for the indicator taking the value 1 when vertex $\Ver_1$ is found by the fake news at time $T_{a_n}^{\sss {\mathcal{F}}}+t$, where $a_n$ will be chosen appropriately later on. Note that, with $G_n$ denoting the realisation of the configuration model including its edge weights,
	\eqn{
	\label{Ntk-repr}
	\barNt=\expec[\indicwo{\sss \Ver_1}(t)\mid G_n].
	}
We will show that $\E[\indicwo{\sss \Ver_1}(t)\mid G_n]\convp \prob(V^{\sss {\mathcal{F}}}\leq t)$ by using a conditional second moment method. Recall that $\SWG(s)$ denotes the SWG started from $U$ at real time $s$ with exploration. We perform the analysis conditionally on $\SWG\left(T_{a_n}^{\sss {\mathcal{F}}}\right):=\Psi_n$, that is, the first-passage percolation exploration graph at the time when it reaches size $a_n$. Here $a_n$ can be taken as $n^\delta$ for some $\delta\in (0,1)$.

To apply the conditional second moment method, first note that
	\eqan{
	\expec[\barNt \mid \Psi_n]&=\prob\big(\indicwo{\sss \Ver_1}(t)=1 \mid \Psi_n)\nn\\
	&=\prob(\mbox{$\Ver_1$ is found by fake news at time $T_{a_n}^{\sss {\mathcal{F}}}+t$}
	\mid \Psi_n\big),\nn
	}
and
	\eqan{
	\expec[(\barNt)^2 \mid \Psi_n]
	&=\prob(\indicwo{\sss \Ver_1}(t)=\indicwo{\sss \Ver_2}(t)=1 \mid \Psi_n)\nn\\
	&=\prob(\mbox{$\Ver_1,\Ver_2$ are found by fake news at time $T_{a_n}^{\sss {\mathcal{F}}}+t$}\mid
	 \Psi_n).\nn
	}
Therefore, it suffices to show that the first factors in the above two right-hand sides converge to $\prob(V^{\sss {\mathcal{F}}}\leq t)$ and $\prob(V^{\sss {\mathcal{F}}}\leq t)^2$, respectively. Indeed, in this case,
	$$
	\expec[\barNt\mid \Psi_n]\convp \prob(V^{\sss {\mathcal{F}}}\leq t),
	$$
while ${\mathrm{Var}}(\barNt\mid \Psi_n)=\op(1)$, so that $\barNt\convp \prob(V^{\sss {\mathcal{F}}}\leq t)$, as required.
\medskip

We can construct $\Psi_n$ by first growing the one-type SWG from vertex $\Ver$ to size $a_n$. The time when this occurs is $T_{a_n}^{\sss {\mathcal{F}}}$, which converges in distribution to $V^{\sss {\mathcal{F}}}$.

Let $X(\Ver\leftrightarrow \Ver_1)$ denote the passage time between vertices $\Ver$ and $\Ver_2$ in a one-type process with only the fake news infection. It follows from the analysis in \cite{BhaHofHoo09b} that $X(\Ver\leftrightarrow \Ver_1)$ converges in distribution to $V^{\sss {\mathcal{F}}}+V_2^{\sss {\mathcal{F}}}$: As described above, we first grow $\SWG(T_{a_n}^{\sss {\mathcal{F}}})$. Then we grow the SWG from $\Ver_1$ until it hits $\SWG(T_{a_n}^{\sss {\mathcal{F}}})$. This occurs at a time that converges in distribution to $V^{\sss {\mathcal{F}}}$ -- indeed, $V^{\sss {\mathcal{F}}}$ describes the asymptotic explosion time for an exploration process started at a uniform vertex. Hence,
	\begin{equation}\label{eq:onetypeconv}
	\prob(X(\Ver\leftrightarrow \Ver_1)\leq T_{a_n}^{\sss {\mathcal{F}}}+t\mid \Psi_n)
	\convp \prob(V^{\sss {\mathcal{F}}}\leq t).
	\end{equation}
In a similar way, we conclude that $\prob(X(\Ver\leftrightarrow \Ver_1),X(\Ver\leftrightarrow \Ver_2)\leq T_{a_n}^{\sss {\mathcal{F}}}+t\mid \Psi_n)\convp \prob(V^{\sss {\mathcal{F}}}\leq t)^2.$ The second moment method now proves \eqref{Ntk-conv-FPP}.
\qed
\medskip

We are now ready to complete the proof of Theorem \ref{thm-strong-CM-explos}:

\begin{proof}[Proof of Theorem \ref{thm-strong-CM-explos}]
By \eqref{eq:feasibility}, there must be an $\vep>0$ such that $\P(L^{\sss \mathcal{F}} < L^{\sss \mathcal{R}}-\vep) > 0$. Further, since $V^{\sss {\mathcal{F}}}$ is a continuous non-negative random variable with $\inf {\rm supp}(V^{\sss {\mathcal{F}}})=0$, also $\P(V^{\sss {\mathcal{F}}}<\vep/2)>0$. If the traversal time for the correct news of {\em all} edges incident to $\Ver$ satisfies $L^{\sss \mathcal{F}} < L^{\sss \mathcal{R}}-\vep$, then the correct news will, with positive probability, not have found any vertices at time $V^{\sss {\mathcal{F}}}+\vep/2$. Thus, it cannot block the fake news before time $T_{a_n}^{\sss {\mathcal{F}}}+\vep/2$, while, by Proposition \ref{prop-degree-time-FPP}, the fake news will, with (conditional) probability converging to 1, have reached a positive proportion of the vertices. This proves the strong survival in Theorem \ref{thm-strong-CM-explos}.
\end{proof}
\medskip

\paragraph{\bf Acknowledgement.}
The work of RvdH is supported by the Netherlands Organisation for Scientific Research (NWO) through VICI grant 639.033.806 and the Gravitation {\sc Networks} grant 024.002.003. This project was initiated at the workshop in honour of the 85th birthday of Alexander Borovkov at Novosibirsk State University. We thank the organisors for bringing us together.

\bibliographystyle{plain}
\def\cprime{$'$}


\begin{thebibliography}{10}

\bibitem{AhlDeiJan17}
D.~Ahlberg, M.~Deijfen, and S.~Janson.
\newblock Competing first passage percolation on random graphs with finite
  variance degrees.
\newblock (2017).

\bibitem{AidHZ13}
E.~A\"{\i}d\'ekon, Y.~Hu and O.~Zindy.
\newblock{The precise tail behavior of the total progeny of a killed branching random walk.} 
\newblock{\em Ann. Appl. Probab.}, {\bf 41}(6):3786--3878, (2013).

\bibitem{AllGen17}
H.~Allcott and M.~Gentzkow.
\newblock Social media and fake news in the 2016 election.
\newblock {\em Journal of Economic Perspectives}, {\bf 31}(2):211--36, (2017).

\bibitem{AidJaf11}
E.~A\"{\i}d\'ekon and B.~Jaffuel.
\newblock Survival of branching random walks with absorption.
\newblock {\em Stochastic Process. Appl.}, {\bf 121}(9):1901--1937, (2011).

\bibitem{BarHofKom15}
E.~Baroni, R.~van~der Hofstad, and J.~Komj\'athy.
\newblock Fixed speed competition on the configuration model with infinite
  variance degrees: unequal speeds.
\newblock {\em Electron. J. Probab.}, {\bf 20}:Paper No. 116, 48, (2015).

\bibitem{BarHofKom17}
E.~Baroni, R.~van~der Hofstad, and J.~Komj\'athy.
\newblock Nonuniversality of weighted random graphs with infinite variance
  degree.
\newblock {\em J. Appl. Probab.}, {\bf 54}(1):146--164, (2017).

\bibitem{Bham08}
S.~Bhamidi.
\newblock {First passage percolation on locally treelike networks. {I}. {D}ense
  random graphs}.
\newblock {\em Journal of Mathematical Physics}, {\bf 49}:125218, (2008).

\bibitem{BhaHofHoo09a}
S.~Bhamidi, R.~van~der Hofstad, and G.~Hooghiemstra.
\newblock Extreme value theory, {P}oisson-{D}irichlet distributions, and first
  passage percolation on random networks.
\newblock {\em Adv. in Appl. Probab.}, {\bf 42}(3):706--738, (2010).

\bibitem{BhaHofHoo09b}
S.~Bhamidi, R.~van~der Hofstad, and G.~Hooghiemstra.
\newblock First passage percolation on random graphs with finite mean degrees.
\newblock {\em Ann. Appl. Probab.}, {\bf 20}(5):1907--1965, (2010).

\bibitem{BhaHofHoo12}
S.~Bhamidi, R.~van~der Hofstad, and G~Hooghiemstra.
\newblock First passage percolation on the {E}rd{\H o}s-{R}\'enyi random graph.
\newblock {\em Combin. Probab. Comput.}, {\bf 20}(5):683--707, (2011).

\bibitem{BhaHofKom14}
S.~Bhamidi, R.~van~der Hofstad, and J.~Komj\'athy.
\newblock The front of the epidemic spread and first passage percolation.
\newblock {\em J. Appl. Probab.}, 51A(Celebrating 50 Years of The Applied
  Probability Trust):101--121, 2014.

\bibitem{BhaHofHoo17}
S.~Bhamidi, R.~van~der Hofstad, and G.~Hooghiemstra.
\newblock Universality for first passage percolation on sparse random graphs.
\newblock {\em Ann. Probab.}, {\bf 45}(4):2568--2630, (2017).

\bibitem{Bigg76}
J.~D. Biggins.
\newblock The first- and last-birth problems for a multitype age-dependent
  branching process.
\newblock {\em Advances in Appl. Probability}, {\bf 8}(3):446--459, (1976).

\bibitem{BigLubShwWei91}
J.~D. Biggins, B.~Lubachevsky, A.~Shwartz, and A.~Weiss.
\newblock A branching random walk with a barrier.
\newblock {\em Ann. Appl. Probab.}, {\bf 1}(4):573--581, (1991).

\bibitem{Boll80b}
B.~Bollob{\'a}s.
\newblock A probabilistic proof of an asymptotic formula for the number of
  labelled regular graphs.
\newblock {\em European J. Combin.}, {\bf 1}(4), 311--316, (1980).

\bibitem{Bordenave08}
C.~Bordenave.
\newblock On the birth-and-assassination process, with an application to scotching
a rumor in a network. 
\newblock {\em Electron. J. Probab.}, {\bf 13}, 2014–-2030, (2008).

\bibitem{ConRubChe15}
N.~Conroy, V.~Rubin, and Y.~Chen.
\newblock Automatic deception detection: Methods for finding fake news.
\newblock {\em Proceedings of the Association for Information Science and
  Technology}, {\bf 52}(1):1--4, (2015).

\bibitem{DeiHof16}
M.~Deijfen and R.~van~der Hofstad.
\newblock The winner takes it all.
\newblock {\em Ann. Appl. Probab.}, {\bf 26}(4):2419--2453, (2016).

\bibitem{DeiHofSfr22}
M.~Deijfen, R.~van~der Hofstad and M.~Sfragara. 
\newblock The winner takes it all but one. 
\newblock arXiv:2204.04125 [math.PR] (2022).

\bibitem{DS2013}
D.~Denisov, and S.~Shneer. 
\newblock Asymptotics for the first passage times of Levy processes and random walks. 
\newblock {\em J. Appl. Probab.}, {\bf 50}(1):64-84, (2013).

\bibitem{Hag98}
O.~H{\"a}ggstr{\"o}m, and R.~Pemantle. 
\newblock First passage percolation and a model for competing spatial
growth. 
\newblock {\em J. Appl. Probab.}, {\bf 35}(3):683--692, (1998).

\bibitem{Harr63}
T.~Harris.
\newblock {\em The theory of branching processes}.
\newblock Die Grundlehren der Mathematischen Wissenschaften, Bd. 119.
  Springer-Verlag, Berlin, (1963).

\bibitem{Hofs17}
R.~van~der Hofstad.
\newblock {\em Random graphs and complex networks. {V}olume 1}.
\newblock Cambridge Series in Statistical and Probabilistic Mathematics.
  Cambridge University Press, Cambridge, (2017).

\bibitem{Hofs18}
R.~van~der Hofstad.
\newblock {\em Random graphs and complex networks. {V}olume {2}}.
\newblock (2023).
\newblock In preparation, see\\{\tt
  https://www.win.tue.nl/$\sim$rhofstad/NotesRGCNII.pdf}.

\bibitem{Hofs20}
R.~van~der Hofstad.
\newblock {\em Stochastic processes on random graphs}.
\newblock (2023+).
\newblock In preparation, see\\{\tt
  https://www.win.tue.nl/$\sim$rhofstad/SaintFlour\_SPoRG.pdf}.
  
\bibitem{Hofs21b}
R.~van~der Hofstad.
\newblock {\em The giant in random graphs is almost local}.
\newblock arXiv:2103.11733 [math.PR], (2021).

\bibitem{HofLeeSte17}
R.~van~der Hofstad, J.H.S.~van Leeuwaarden, and C.~Stegehuis.
\newblock Hierarchical configuration model.
\newblock {\em Internet Math.} {\bf 1} ar{X}iv:1512.08397 [math.{PR}], (2017).

\bibitem{Hol21}
M.~Hoeller. 
\newblock The human component in social media and fake news: the performance of UK opinion leaders on Twitter during the Brexit campaign.
\newblock {\em European Journal of English Studies}, {\bf 25:1}, 80--95, (2021).

\bibitem{JagNer84}
P.~Jagers and O.~Nerman.
\newblock The growth and composition of branching populations.
\newblock {\em Adv. in Appl. Probab.}, {\bf 16}(2):221--259, (1984).

\bibitem{Jans09b}
S.~Janson.
\newblock Susceptibility of random graphs with given vertex degrees.
\newblock {\em J. Comb.}, {\bf 1}(3-4):357--387, (2010).

\bibitem{JorKom20}
J.~Jorritsma and J.~Komj\'{a}thy.
\newblock {\em Distance evolutions in growing preferential attachment graphs}.
\newblock {\em Ann. Appl. Probab.}, {\bf 32}(6): 4356--4397, (2022).

\bibitem{HofKom15}
J.~Komj\'athy.
\newblock Fixed speed competition on the configuration model with infinite
  variance degrees: equal speeds.
\newblock Available at ar{X}iv:1503.09046 [math.{PR}], Preprint (2015).

\bibitem{Kord05}
G,.~Kordzakhia.
\newblock{The escape model on a homogeneous tree.}
\newblock{\em Elect. Comm. in Probab.}, {\bf 10}:113--124, (2005).

\bibitem{KordLa05}
G.~Kordzakhia, and S.~P.~Lalley. 
\newblock A two-species competition model on d. 
\newblock {\em Stochastic Process. Appl.}, {\bf 115}(5):781--796, (2005).

\bibitem{Kort16}
I.~Kortchemski.
\newblock{Predator-Prey Dynamics on Infinite Trees:
A Branching Random Walk Approach.}
\newblock {\em J. Theor. Probab.}, {\bf 29}:1027--1046, (2016).

\bibitem{King75}
J.~F.~C. Kingman.
\newblock The first birth problem for an age-dependent branching process.
\newblock {\em Ann. Probability}, {\bf 3}(5):790--801, (1975).

\bibitem{Lazer_etal18}
D.~Lazer, M.~Baum, Y.~Benkler, A.~Berinsky, K.~Greenhill, F.~Menczer,
  M.~Metzger, B.~Nyhan, G.~Pennycook, D.~Rothschild, et~al.
\newblock The science of fake news.
\newblock {\em Science}, {\bf 359}(6380):1094--1096, (2018).

\bibitem{LyoPemPer95}
R.~Lyons, R.~Pemantle, and Y.~Peres.
\newblock Conceptual proofs of {$L\log L$} criteria for mean behavior of
  branching processes.
\newblock {\em Ann. Probab.}, {\bf 23}(3):1125--1138, (1995).

\bibitem{Samu71}
M.~Samuels.
\newblock {Distribution of the branching-process population among generations}.
\newblock {\em Journal of Applied Probability}, {\bf 8}(4):655--667, 1971.

\bibitem{SteHofLee16a}
C.~Stegehuis, R.~van~der Hofstad, and J.S.H~van. Leeuwaarden.
\newblock Epidemic spreading on complex networks with community structures.
\newblock {\em Scientific Reports}, {\bf 6}:29748, (2016).

\bibitem{VosRoyAra18}
S.~Vosoughi, D.~Roy and S.~Aral 
\newblock The spread of true and false news online. 
\newblock {\em Science}, {\bf 359(6380)}, 1146--1151, 2018.

\bibitem{Zwart01}
B.~Zwart.
\newblock Tail Asymptotics for the Busy Period in the GI/G/1 Queue.
\newblock {\em Mathematics of Operations Research}, {\bf 26}(3):485--493, (2001).

\end{thebibliography}

\end{document}